\documentclass[preprint]{elsarticle}

\journal{Nonlinear Analysis: Hybrid Systems}

\usepackage{enumerate}
\usepackage{graphicx}
\usepackage{caption}
\usepackage{subcaption}
\usepackage{amsmath, amssymb,amsthm}

\usepackage{framed}

\usepackage{url}

\newtheorem{thm}{Theorem}
\newtheorem{lem}{Lemma}
\newtheorem{cor}{Corollary}

\newtheorem{alg}{Algorithm}
\newtheorem{assumption}{Assumption}
\newtheorem{definition}{Definition}
\newtheorem{problem}{Problem}

\usepackage[usenames,dvipsnames]{color}

\usepackage{ifthen,version}

\newboolean{color-new}
\setboolean{color-new}{false}
\newcommand{\new}[1]{\ifthenelse{\boolean{color-new}}%
  {{\color{blue} #1}}{#1}}

\DeclareMathAlphabet{\mathpzc}{OT1}{pzc}{m}{it}
\DeclareFontFamily{U}{msb}{}
\DeclareFontShape{U}{msb}{m}{n}{ <5> <6> <7> <8> <9> gen * msbm
<10> <10.95> <12> <14.4> <17.28> <20.74> <24.88> msbm10}{}
\DeclareSymbolFont{AMSb}{U}{msb}{m}{n}
\DeclareMathSymbol{\Reals}{\mathalpha}{AMSb}{'122}
\DeclareMathSymbol{\Naturals}{\mathalpha}{AMSb}{'116}
\DeclareMathSymbol{\Knumbers}{\mathalpha}{AMSb}{'113}
\DeclareMathSymbol{\Rationals}{\mathalpha}{AMSb}{'121}
\DeclareSymbolFont{AMSb}{U}{msb}{m}{n}
\DeclareMathSymbol{\setB}{\mathalpha}{AMSb}{'102}
\DeclareMathSymbol{\setC}{\mathalpha}{AMSb}{'103}
\DeclareMathSymbol{\setD}{\mathalpha}{AMSb}{'104}
\DeclareMathSymbol{\setE}{\mathalpha}{AMSb}{'105}
\DeclareMathSymbol{\setF}{\mathalpha}{AMSb}{'106}
\DeclareMathSymbol{\setI}{\mathalpha}{AMSb}{'111}
\DeclareMathSymbol{\setK}{\mathalpha}{AMSb}{'113}
\DeclareMathSymbol{\setM}{\mathalpha}{AMSb}{'115}
\DeclareMathSymbol{\setN}{\mathalpha}{AMSb}{'116}
\DeclareMathSymbol{\setP}{\mathalpha}{AMSb}{'120}
\DeclareMathSymbol{\setQ}{\mathalpha}{AMSb}{'121}
\DeclareMathSymbol{\setR}{\mathalpha}{AMSb}{'122}
\DeclareMathSymbol{\setS}{\mathalpha}{AMSb}{'123}
\DeclareMathSymbol{\setT}{\mathalpha}{AMSb}{'124}
\DeclareMathSymbol{\setU}{\mathalpha}{AMSb}{'125}
\DeclareMathSymbol{\setV}{\mathalpha}{AMSb}{'126}
\DeclareMathSymbol{\setW}{\mathalpha}{AMSb}{'127}
\DeclareMathSymbol{\setX}{\mathalpha}{AMSb}{'130}
\DeclareMathSymbol{\setY}{\mathalpha}{AMSb}{'131}
\DeclareMathSymbol{\setZ}{\mathalpha}{AMSb}{'132}

\begin{document}

\begin{frontmatter}

\title{Projection-Based Iterative Mode Scheduling for Switched Systems} %

\author{T. M. Caldwell\corref{mycorrespondingauthor}}\ead{caldwelt@u.northwestern.edu}
\author{T. D. Murphey}\ead{t-murphey@northwestern.edu}
\address{Mechanical Engineering\\Northwestern University\\Evanston IL 60208}
\cortext[mycorrespondingauthor]{Corresponding author}

\begin{abstract}
This paper describes a method for scheduling the events of a switched system to achieve an optimal performance.  The approach has guarantees on convergence and computational complexity that parallel derivative-based iterative optimization but in the infinite dimensional, integer constrained setting of mode scheduling.  In comparison to methods relying on mixed integer programming, the presented approach does not require a priori discretizations of time or state.  Furthermore, in comparison to embedding and relaxation methods, every iteration of the algorithm returns a dynamically feasible solution.  A large class of problems call for optimal mode scheduling. 
\new{This paper considers a vehicle tracking problem and a high dimensional multimachine power network synchronization problem. For the power network example, }
both single horizon and receding horizon approaches prevent instability of the network, and the receding horizon approach does so at near real-time speeds on a single processor.
\end{abstract}


\end{frontmatter}


\section{Introduction}
Optimally scheduling the distinct dynamic modes of a switched system is not a combinatoric problem when using variational techniques.  Instead of discretizing in state or time a priori and applying mixed integer programming or combinatoric searches \cite{bemporad_borrelli_morari,caines_shaikh_2005,gorges_izak_liu,seatzu_corona_etal}, we remove integer constraints and calculate local variations in the resulting unconstrained set.  The locally varied solution is projected back to the set of dynamically feasible trajectories.  This work is an extension from \cite{caldwell_murphey_acc12,caldwell_murphey_cdc12} where we show that iteratively stepping in the direction of the local variation and projecting the result is absolutely continuous in the step size.   That absolute continuity result is needed for the Armijo-like line search presented here to have a sufficient descent property for convergence.
\new{This paper is a more complete presentation of \cite{caldwell_murphey_acc12,caldwell_murphey_cdc12}, as well as our preliminary optimal mode scheduling work in \cite{caldwell_murphey_cdc13}.
}

\new{
There are two main themes to numerical mode scheduling: averaging methods \cite{bengea_decarlo, vasudevan_etal}, and insertion methods \cite{axelsson_wardi_et_al,egerstedt_wardi_axelsson,gonzalez_vasudevan_etal,wardi_egerstedt_hale,wardi_egerstedt_acc12}. Averaging methods directly handle chattering of optimal mode scheduling solutions by modulating an embedding of the control authority by averaging the time spent in each mode. Insertion methods use needle variations to iteratively update the control schedule in a manner that both reduces the cost and guarantees convergence to a local infima. As seen in the examples in \cite{wardi_egerstedt_hale} and here, insertion methods rapidly descend the cost function in a few iterations when starting from an initial guess far from a local infima.

Our approach is an insertion method. We apply a projection operator in a similar manner to gradient projection methods for finite dimensional inequality constrained optimization \cite{nocedal_wright} and optimal control of trajectory functionals \cite{hauser,hausermeyer}.  We locally search an unconstrained space quantifying the search through a cost function composed with the projection. As such, our approach always returns a dynamically feasible solution at each step of the optimization. In comparison to embedding methods, we never relax the problem statement.

Like insertion methods \cite{axelsson_wardi_et_al,egerstedt_wardi_axelsson,gonzalez_vasudevan_etal,wardi_egerstedt_hale,wardi_egerstedt_acc12}, we use the mode insertion gradient to iteratively update the switching control (see \cite{axelsson_wardi_et_al,egerstedt_wardi_axelsson,gonzalez_vasudevan_etal,wardi_egerstedt_hale,wardi_egerstedt_acc12} and our review in Section~\ref{sec-MIG} for a description of the mode insertion gradient). In comparison, we utilize the mode insertion gradient differently.  Other insertion methods use the mode insertion gradient to determine an insertion time and mode and conduct a line search on the insertion duration, or corresponding Lebesgue measure, to update the switching control. This approach was limited to a single mode insertion per iteration in \cite{gonzalez_vasudevan_etal}, but \cite{wardi_egerstedt_hale,wardi_egerstedt_acc12} extended this approach by conducting a line search on the full Lebesgue measure of many possible insertions in a novel way. Our approach is most similar to \cite{wardi_egerstedt_hale,wardi_egerstedt_acc12} but it differs in that we treat the mode insertion gradient as Lebesgue integrable curves that act as local variations in the unconstrained space and project steps in the direction of the local variation to feasible switching controls. While this process is similar to \cite{wardi_egerstedt_hale,wardi_egerstedt_acc12} in that it results in a change of Lebesgue measure, our line search steps directly in the unconstrained control space while \cite{wardi_egerstedt_hale,wardi_egerstedt_acc12} steps the line search in the feasible control space. Whereas our line search consists of choosing a distance and projecting, \cite{wardi_egerstedt_hale,wardi_egerstedt_acc12} require additional algorithmic complexity to determine the switching control for the desired Lebesgue measure. Additionally, the equations in Section~\ref{sec-suffdesc} for the line search include both the mode insertion gradient and projection, making the role of the curvature (and higher order derivatives) of the mode insertion gradient explicit.

The mode insertion gradient as a local variation is not an actual gradient since the set of valid variations that guarantee a feasible projection do not form a Hilbert space.  Despite this fact, our projection-based mode scheduling does parallel iterative optimization techniques based on differentiability; each iterate steps in a descent direction of size given by backtracking as part of a line search, which sufficiently reduces the cost for guarantees on convergence.  The primary objective of this paper is to show that the same procedures from derivative-based optimization are valid for mode scheduling even though mode scheduling is an infinite dimensional, non-smooth problem.  In other words, in the context of mode scheduling, we prove convergence properties for the procedure of: 1) calculating a descent direction, 2) taking a step of size calculated using backtracking, 3) updating, and repeating.
}

Explicitly, the contributions of this paper are the following.
\begin{enumerate}[(A)]
\item A local approximation of the cost function in the direction of the negative mode insertion gradient for use in iterative optimization. 
\item Showing the negative mode insertion gradient is a descent direction for the projection-based framework.
\item A test for sufficient descent.
\item Showing conditions under which backtracking calculates a step size which satisfies sufficient descent in a finite number of iterations.
\end{enumerate}
The local approximation of the cost (A) in the direction of the local variation is needed to prove the (C), and (D) contributions, while (B) is a consequence of our choice of projection. Our analysis concludes by finding guarantees for convergence.

This paper is organized as follows: Section \ref{sec-review} reviews switched systems, the projection-based optimal mode scheduling problem, the switching time gradient, and the mode insertion gradient. Section~\ref{sec-approx_cost_ST} provides and analyzes the local approximation of the cost, which is Contribution A. Section~\ref{sec-approx_cost_ST} additionally shows that the negative mode insertion gradient is a descent direction, Contribution B.  Section~\ref{sec-suffdesc} proposes both a sufficient descent condition and a backtracking algorithm and proves convergence, Contributions C and D.  Section~\ref{sec-mode_sched} presents the full mode scheduling algorithm and discusses implementation issues. Finally, Section~\ref{sec-eg} applies mode scheduling to two examples. The first is a vehicle tracking problem with control authority constrained to four modes. The second is a disturbance response for the IEEE 118 Bus Test Case multimachine power network \cite{christie}, which is composed of 54 generators, 108 states, 118 buses and 186 lines. 

A notation table is in the Appendix and unless otherwise stated the lemma proofs are in the Appendix.

\section{Review and Introductory Results}
\label{sec-review}
The following reviews switching control of switched systems \cite{caldwell_murphey_acc12,caldwell_murphey_cdc12}, the switching time gradient \cite{caldwell_murphey_journal,egerstedt_wardi_axelsson,johnson_murphey_tac,xu_antsaklis_2}, the mode insertion gradient \cite{egerstedt_wardi_axelsson,gonzalez_vasudevan_etal,wardi_egerstedt_hale,wardi_egerstedt_acc12}, the max-projection operator for switched systems \cite{caldwell_murphey_acc12,caldwell_murphey_cdc12}, and the projection-based optimal mode scheduling problem \cite{caldwell_murphey_acc12,caldwell_murphey_cdc12}.  


\subsection{Switched Systems}
The evolution of a switched system over the bounded time interval $[0,T]$, $T>0$ depends on a set of distinct modes. Supposing there are $N$ such modes describing the system's motion, label them $f_i:\setR^n\rightarrow\setR$, $i\in\{1,\ldots,N\}$. At any time $t\in[0,T]$, the evolution depends only on a single mode---i.e. $\dot{x}(t) = f_i(x(t))$ for one $i\in \{1,\ldots,N\}$. The times when the system transitions from one mode to another are referred to as \emph{switching times} and are the times $t\in[0,T]$ for which $\dot{x}(t^-) = f_i(x(t^-))$ but $\dot{x}(t^+) = f_j(x(t^+))$ where $i\neq j\in \{1,\ldots,N\}$ and `$\cdot^+$' is the limit from the right and `$\cdot^-$' is the limit from the left.  Label $M-1$ as the total number of switching times.

We consider two equivalent representations to parameterize a switched system: mode schedule, and switching control.  Both representations play a critical role in the mode scheduling algorithm. The mode schedule is a natural way to specify the control policy, while the variations for numerical iterative descent depend on the switching control representation. 

The mode schedule specifies which mode dictates the system evolution at any given time. 
\begin{definition}\label{def-sched}
A \emph{mode schedule} is composed of the pair $(\Sigma,\mathcal{T})$ where $\Sigma = [\sigma_1,\ldots,\sigma_M]$ is the mode sequence and $\mathcal{T} = [T_1,\ldots,T_{M-1}]$ is the strictly monotonically increasing vector of switching times.  Each mode is $\sigma_i\in\{1,\ldots,N\}$, each switching time is $T_i\in[0,T]$, and the total number of elements in the mode sequence is $M\in\setN$.
\end{definition}
The switching control signal $u$ defines a mode schedule through a piecewise constant signal.  This representation enables taking variations in the control, including changing the order of modes.
\begin{definition}\label{def-omega}
The curve $u = [u_1,\ldots,u_N]^T$ composed of $N$ piecewise constant functions of time is a \emph{switching control} if 
\begin{itemize}
\item for almost every $t\in[0,T]$, $\sum_{i=1}^N u_i(t) = 1$, and 
\item for each $i\in\{1,\ldots,N\}$, $u_i(t)\in\{0,1\}$.
\end{itemize}
\end{definition}

A unique mapping exists between both representations.  Given a mode schedule, $(\Sigma,\mathcal{T})$, the switching control $u$ is $u(t) = e_{\sigma_i}$ for $t\in[T_{i-1},T_i)$, $i=1,\ldots,M$ where $T_0 = 0$, $T_M = T$ and $e_{\sigma_i}$ is the $\sigma_i^{\textrm{th}}$ column of the $N$ dimensional identity matrix. Additionally, given a switching control $u$, the mode schedule is $(\Sigma,\mathcal{T}) = ([\sigma_1,\ldots,\sigma_M],[T_1,\ldots,T_{M-1}])$ where $\mathcal{T} = \{t\in(0,T)|u(t^+)\neq u(t^-)\}$ and $e_{\sigma_i} = u(t)$ for $t\in[T_{i-1},T_i)$, $i = 1,\ldots,M$.

Each parameterization is identified as chattering or non-chattering depending on the number of switching times. Define \emph{non-chattering} as:

\begin{definition}\label{def-non-chattering}
A mode schedule $(\Sigma,\mathcal{T})$ is \emph{non-chattering} when there is a $\delta t>0$ such that every active mode dwells for at least $\delta t$ time---i.e. $|T_{i} - T_j|>\delta t$ for each adjacent pair $T_{i}$ and $T_j\in\mathcal{T}$. 
\end{definition}

Since the time interval $[0,T]$ is bounded, the number of switching times, $M-1$, is finite. We say a switching control is non-chattering if the corresponding mode schedule is non-chattering. Label the set of all \emph{non-chattering switching controls} as $\Omega$.  Therefore, any $u\in\Omega$ switches values at most a finite number of times in $[0,T]$. We embed $\Omega$ in the space of Lebesgue integrable functions from $[0,T]$ to $\setR^N$, labelled $\mathcal{U}$, for performing addition and scalar multiplication operations.

A \emph{switched system trajectory} is the state and the switching control,  $(x,u)$---alternatively, $(x,\Sigma,\mathcal{T})$---that satisfies the state equations.  Here, $x$ is assumed to be an element of $\mathcal{X}$, the space of Lebesgue integrable functions from $[0,T]$ to $\setR^n$. If, as above, the switched system has state $x$ and switching control $u$, then the state equations are
\begin{equation}
\dot{x}(t) = F(x(t),u(t)) := \sum_{i=1}^Nu_i(t)f_i(x(t)),\hspace{10pt}x(0) = x_0.
\label{eq-state}
\end{equation}
The corresponding formal definition of a switched system trajectory is:
\begin{definition}\label{def-S}
The pair $(x,u)\in\mathcal{X}\times\mathcal{U}$ is a \emph{feasible switched system trajectory} if
\begin{itemize}
\item $u\in\Omega$ (i.e. $u$ is a non-chattering switching control) and
\item $x(t)-x(0)-\int_0^tF(x(\tau),u(\tau))d\tau = 0$ for almost all $t\in[0,T]$.\footnote{The integral is the Lebesgue integral.}
\end{itemize}
Denote the set of all such pairs of state and switching controls by $\mathcal{S}$.
\end{definition}


\subsection{Optimal Mode Scheduling Problem}
Define the objective function as 
\[
J(x,u) = \int_0^T\ell(x(\tau))d\tau
\]
where the running cost, $\ell:\setR^n\rightarrow\setR$ is $\mathcal{C}^2$ in $\setR^n$.  (We include the control $u\in\mathcal{U}$ in the definition of the objective because it is a design variable in the following optimization problem.)
\new{\begin{problem}
\label{prob-const}
Find the infimum of the cost $J$ with respect to $x$ and $u$ under the constraint that $x$ and $u$ constitute a feasible switched system trajectory\textemdash i.e. $(x,u)\in\mathcal{S}$:
\[
\inf_{(x,u)\in\mathcal{S}}J(x,u).
\]
\end{problem}
Sequences of non-chattering switching controls, $\{u^k\}\in\Omega$, can converge to chattering switching controls \cite{bengea_decarlo} necessitating pursuing an infimum. We tackle the problem by generating a sequence of non-chattering switched system trajectories, $\{(x^0,u^0),(x^1,u^1),(x^2,u^2),\ldots\}$ for which $\lim_{k\rightarrow \infty}J(x^k,u^k) = J^\star$. 
}

The mode scheduling algorithm in this paper provides a method for generating such a sequence and gives conditions which guarantee that $J(x^k,u^k)$ converges to $J^\star$ (while, at every iteration, $(x^k,u^k)\in\mathcal{S}$).  Since the proposed approach is an iterative descent one, the infimum $J^\star$ might not be the cost's global lower bound. Therefore, solutions to Problem~\ref{prob-const} can only be argued as locally infimal in general, which is often the case for non-convex iterative optimization. 


\subsection{Switching Time Gradient}
The problem of optimizing the switching times when the mode sequence is fixed is considered in \cite{caldwell_murphey_journal,egerstedt_wardi_axelsson,flasskamp_murphey_oberblobaum,johnson_murphey_tac,xu_antsaklis_2}. 
Suppose for a given switching control, $u$, the corresponding mode schedule is $(\Sigma,\mathcal{T})$. 
\new{
Additionally, suppose the cost associated with $(\Sigma,\mathcal{T})$ is $J(\Sigma,\mathcal{T}):=\int_0^T\ell(x(\tau))d\tau$, where $x$ is the solution to the state equations Eq.~(\ref{eq-state}), and 
$\ell$ and $f_i$, $i\in\{1,\ldots,N\}$ are $\mathcal{C}^1$ in $\setR^n$. 
}
Then, the $i^\textrm{th}$ switching time derivative of the cost is (\cite{caldwell_murphey_journal,egerstedt_wardi_axelsson,flasskamp_murphey_oberblobaum,johnson_murphey_tac,xu_antsaklis_2})
\begin{equation}
\frac{\partial}{\partial T_i}J(x,u) = \rho^T(T_i)(f_{\sigma_i}(x(T_i))-f_{\sigma_{i+1}}(x(T_i)))
\label{eq-stgrad}
\end{equation}
where $x$ is the solution to the state equations, Eq.~(\ref{eq-state}), and $\rho$ is the solution to the following adjoint equation\footnote{$D$ is the partial derivative with respect to the only argument. When a function has multiple arguments, the argument slot is specified. For example, $Dg(a) = \frac{\partial}{\partial a} g(a)$, $D_1g(a,b) = \frac{\partial}{\partial a} g(a,b)$, and $D_2g(a,b) = \frac{\partial}{\partial b} g(a,b)$.}
\begin{equation}
\begin{array}{c}
\dot{\rho}(t) = -Df_{\sigma_i}(x(t))^T\rho(t)-D\ell(x(t))^T,\\T_{i-1}<t<T_{i}\hspace{10pt}\textrm{for }i\in\{1\ldots,M\}
\end{array}
\label{eq-rhodot}
\end{equation}
where $\rho(T) = 0$. We call $\frac{\partial}{\partial \mathcal{T}}J(x,u):=[\frac{\partial}{\partial T_1}J(x,u),\ldots,\frac{\partial}{\partial T_{M-1}}J(x,u)]^T$ the switching time gradient \cite{caldwell_murphey_journal,egerstedt_wardi_axelsson,flasskamp_murphey_oberblobaum,johnson_murphey_tac,xu_antsaklis_2}.  The adjoint equation Eq.~(\ref{eq-rhodot}) plays an important role in the mode insertion gradient, discussed next.


\subsection{Mode Insertion Gradient}
\label{sec-MIG}
For projection-based switched system optimization, the cost does not have a gradient in the same sense that differentiable functions in an inner product space have a gradient.  However, the cost does have a function with a similar role in the optimization as the gradient plays in finite dimensional smooth optimization.  This function is the \emph{mode insertion gradient} \cite{egerstedt_wardi_axelsson,gonzalez_vasudevan_etal,wardi_egerstedt_hale,wardi_egerstedt_acc12}. 

\new{
The mode insertion gradient measures the cost's sensitivity to inserting a mode at a time $t\in[0,T]$ for an infinitesimal interval. Suppose the current switched system is $(x,u)\in\mathcal{S}$ where $u$ corresponds to schedule $(\Sigma, \mathcal{T})$ such that $\Sigma = [\sigma_1,\ldots,\sigma_M]$ and $\mathcal{T} = [T_1,\ldots,T_{m-1}]$. Let $(\Sigma_{\lambda,a,t},\Sigma_{\lambda,a,t})$ be the mode schedule after inserting mode $a\in\{1,\ldots,N\}$ at time $t\in[0,T]$ for short time duration $\lambda>0$. Supposing the insertion occurs between switching times $T_{i-1}$ and $T_i$, $i\in\{1,\ldots,M\}$, then $\Sigma_{\lambda,a,t} = [\sigma_1\ldots,\sigma_i,a,\sigma_i,\ldots,\sigma_M]$ and $\mathcal{T}_{\lambda,a,t} = [T_1,\ldots,T_{i-1},t, t+\lambda,T_i,\ldots,T_{m-1}]$. The mode $a$ insertion gradient at time $t$, labelled $d_a(t)$, is defined as 
\begin{equation}
\label{eq-insertion_gradient_def}
d_a(t) := \left.\frac{\partial}{\partial \lambda} J(\Sigma_{\lambda,a,t},\mathcal{T}_{\lambda,a,t})\right|_{\lambda\rightarrow 0^+} = \left. D_2J(\Sigma_{\lambda,a,t},\mathcal{T}_{\lambda,a,t})\frac{\partial}{\partial \lambda} \mathcal{T}_{\lambda,a,t}\right|_{\lambda\rightarrow 0^+}
\end{equation}
where $D_2J(\Sigma_\lambda,\mathcal{T}_\lambda)$ is the switching time gradient, Eq.~(\ref{eq-stgrad}) and as such, 
\[
d_a(t) = \rho(t+\lambda)^T (f_{a}(x(t+\lambda))-f_{\sigma_{i}}(x(t+\lambda)))|_{\lambda\rightarrow 0^+}
\]
where $\rho$ is the solution to the adjoint equation Eq.~(\ref{eq-rhodot}). Since $F(x(t),u(t)) = f_{\sigma_{i+1}}(x(t))$,
\begin{equation}
d_a(t) := \rho(t)^T(f_a(x(t)) - F(x(t),u(t))).
\label{eq-insertion_gradient}
\end{equation}
By the definition of the mode $a$ insertion gradient, Eq.~(\ref{eq-insertion_gradient_def}), if $d_a(t)<0$, then there is a $\lambda>0$ such that $J(\Sigma_{\lambda,a,t},\mathcal{T}_{\lambda,a,t}) < J(\Sigma, \mathcal{T})$. Thus, the condition $d_a(t)<0$ guarantees descent in cost is possible, specifically by inserting mode $a$ for a sufficiently short duration at time $t$. 
}

Since the mode $a$ insertion gradient can be calculated for each $t\in[0,T]$ and mode $a\in\{1,\ldots,N\}$, define $d:[0,T]\rightarrow\setR^N$ as the 
\new{\emph{mode insertion gradient} of $J$ at $(x,u)$.}\footnote{In this paper the mode insertion gradient is $d$, an $N$-dimensional list of curves, while in \cite{egerstedt_wardi_axelsson,gonzalez_vasudevan_etal,wardi_egerstedt_hale,wardi_egerstedt_acc12} the mode insertion gradient is $d_a(t)$, the evaluation of $d$ for the $a^\textrm{th}$ mode at time $t$.} That is, $d(t) = [d_1(t),\ldots,d_N(t)]^T$.  

Define the mode insertion gradient difference of mode $a\in\{1,\ldots,N\}$ with that of mode $b\in\{1,\ldots,N\}$ as $d_{ab}(t):=d_a(t)-d_b(t) = \rho(t)^T(f_a(x(t)) - f_b(x(t)))$.\footnote{We use the double index as it is here to represent the difference of the first index with the second throughout the paper.}
In Section~\ref{sec-type2_suffdesc}, the proof of sufficient descent relies on the assumption that $\ddot{d}_{ab}(t):=\ddot{d}_a(t)-\ddot{d}_b(t)$ is Lipschitz continuous over any time interval bounded by adjacent switching times. In order to make such a claim, we make the following assumptions on each vector field $f_i$ and the running cost $\ell$:
\begin{assumption}
\label{ass-C2bnd}
Assume for every $x(t)\in\setR^n$
\begin{enumerate}
\item for each $i\in\{1,\ldots,N\}$, $f_i(x(t))$ is $\mathcal{C}^2$ and there exists $K_2>0$ such that $\|D^2f_i(x(t))\|\leq K_2$, and
\item $\ell(x(t))$ is $\mathcal{C}^2$ and there exists $\overline{K}_2>0$ such that $\|D^2\ell(x(t))\|\leq\overline{K}_2$.
\end{enumerate}
\end{assumption}
In Assumption~\ref{ass-C2bnd}.1, $\|\cdot\|$ is an operator norm on the space of real $n \times n \times n$ operators, while in Assumption~\ref{ass-C2bnd}.2, $\|\cdot\|$ is an operator norm on the space of real $n \times n$ operators. 

With Assumption~\ref{ass-C2bnd}, we can guarantee the existence and uniqueness of both $x$, the solution to the state equation Eq.~(\ref{eq-state}), and $\rho$, the solution to the adjoint equation Eq.~(\ref{eq-rhodot}), for $u\in\Omega$ using Theorem 3.2 in \cite{khalil}. The existence and uniqueness of $x$ and $\rho$ are useful for proving the following Lemma that guarantees Lipschitz continuity of $\ddot{d}_{ab}(t)$. 

\begin{lem}[Lipschitz condition for $\ddot{d}_{ab}(t)$]
\label{lem-lipschitz}
Suppose $u\in\Omega$ and is constant in the time interval $(\tau_1,\tau_2)$, $\tau_1<\tau_2\in[0,T]$, and $d$ is the mode insertion gradient of $J$ calculated from $(x,u)$.   There exists an $L>0$ such that for each $a\neq b\in\{1,\ldots, N\}$ and $t_1,t_2\in(\tau_1,\tau_2)$, 
\[
|\ddot{d}_{ab}(t_2) - \ddot{d}_{ab}(t_1)| \leq L|t_2-t_1|.
\]
\end{lem}
The proof of Lemma~\ref{lem-lipschitz} is in Appendix~\ref{app-lipschitz}.

From Lemma~\ref{lem-lipschitz}, we see that $\ddot{d}_{ab}$ is piecewise Lipschitz for any $u\in\Omega$. Take some $u\in\Omega$ and let $T_1,\ldots,T_{M-1}$ be the switching times in its mode schedule representation.  Then, $u(t)$ is constant for any $t\in(T_i,T_{i+1})$, $i=0,\ldots,M-1$, and therefore, $\ddot{d}_{ab}$ is Lipschitz over each interval $(T_i,T_{i+1})$. 

We never calculate the Lipschitz constant $L$ but its existence is needed to approximate the cost in the direction of the negative mode insertion gradient (see Section~\ref{sec-approx_cost_ST}) and to provide a region of step sizes for which the line search proposed in Section~\ref{sec-suffdesc} satisfies steepest descent.


\subsection{Projection Operator\label{sec-proj}}
We wish to project curves $(\alpha,\mu) \in \mathcal{X}\times\mathcal{U}$ to $(x,u) \in \mathcal{S}$ so that at every step of an iterative mode scheduling optimization guarantees feasible $(x,u)$. To construct an appropriate choice of projection, we first define the mapping $\mathcal{Q}:\mathcal{U}\rightarrow \mathcal{U}$, where the $i^\textrm{th}$ element of $\mathcal{Q}(\mu(t))$, $\mu\in\mathcal{U}$ is defined as (where $\mu_{ij} = \mu_i-\mu_j$):
\begin{equation}
\mathcal{Q}_i(\mu(t)) := \prod^N_{j\neq i}1(\mu_{ij}(t)).
\label{eq-maxProjStep}
\end{equation}
Here $1:\setR\rightarrow\{0,1\}$ is the step function\textemdash i.e. $1(\mu_{ij}(t)) = 0$ if $\mu_{ij}(t)<0$ and $1(\mu_{ij}(t)) = 1$ if $\mu_{ij}(t)\geq 0$.  Note that this mapping always returns a vector of ones and zeros.  

The purpose of the mapping is to project elements of $\mathcal{U}$ to $\Omega$. However, the mapping will not always return a switching control. For example, suppose $N=2$ and $\mu_1(t)=\mu_2(t)$ for a connected interval of $t$. Then, $\mathcal{Q}(\mu(t)) = [1, 1]^T$ for that interval and thus $\mathcal{Q}(\mu)\not\in\Omega$.  For this reason, we only apply $\mathcal{Q}$ to the subset $\mathcal{R}\subset\mathcal{U}$, where $\mathcal{R}$ is defined as the pre-image $\mathcal{Q}^{-1}(\mathcal{S})$.  

In this paper, $\mu$ will always have the form $\mu = u + \gamma v$ where $u\in\Omega$, $\gamma\in\setR^+$ and $v \in \mathcal{U}$.  With the following assumption on $v$, Lemma 1 in \cite{caldwell_murphey_cdc12} guarantees that $u + \gamma v\in\mathcal{R}$.
\begin{assumption}\label{ass-v_crit}
Assume $v = [v_1,\ldots,v_N]^T\in\mathcal{U}$ is piecewise continuous in $[0,T]$ such that for each $i\neq j \in \{1,\ldots,N\}$, $v_i-v_j$ has a finite number of critical points\footnote{A critical point is a point $t$ of a real valued function $v$ in which either $\dot{v}(t)=0$ or $v$ is not differentiable at $t$.} in $[0,T]$.
\end{assumption}
The choice of $v$ used in this paper is the negative mode insertion gradient $v = -d$.  Since $d_{ab} := d_a - d_b = \rho(t)^T(f_a(x(t))-f_b(x(t)))$, $a\neq b\in\{1,\ldots,N\}$, Assumption~\ref{ass-v_crit} in part requires the modes to be distinct on connected time intervals. It may be possible to design the system and pick a control so that $d_{ab}$ is constant for a connected time interval, which we treat as a degeneracy for the purposes of this paper. \new{Due to the possible scenario that $d$ is such that $\mu = u-\gamma d\not\in\mathcal{R}$, we apply Assumption~\ref{ass-v_crit} to the negative mode insertion gradient, $-d$. }

The max-projection is defined using $\mathcal{Q}$ as:
\begin{definition}\label{def-P}
Take $(\alpha,\mu)\in\mathcal{X}\times\mathcal{R}$. The \emph{max-projection}, $\mathcal{P}:\mathcal{X}\times\mathcal{R}\rightarrow \mathcal{S}$, at time $t\in[0,T]$ is 
\begin{equation}
\mathcal{P}(\alpha(t),\mu(t)):=\left\{\begin{array}{l}
\dot{x}(t) = F(x(t),u(t)),\hspace{10pt} x(0) = x_0 \\
u(t) = \mathcal{Q}(\mu(t)).
\end{array}\right.
\label{eq-maxP}
\end{equation}
\end{definition}
The max-projection is a projection---i.e. $\mathcal{P}(\mathcal{P}(\alpha,\mu)) = \mathcal{P}(\alpha,\mu)$ for all $(\alpha,\mu)\in\mathcal{X}\times\mathcal{R}$---according to Lemma 1 of \cite{caldwell_murphey_acc12}. Notice, since the max-projection does not depend on $\alpha$, we occasionally write $\mathcal{P}(\mu)$.  We include the unconstrained state in the definition in order for $\mathcal{P}$ to be a projection.  (Other projections proposed in \cite{caldwell_murphey_acc12} do depend on $\alpha$.)

\subsection{Projection-Based Optimal Mode Scheduling}
\label{sec-proj_ms}

Problem~\ref{prob-const} provides the mode scheduling optimal control problem where the optimization is constrained to $(x,u) \in \mathcal{S}$.  With the definition of $\mathcal{P}$ in Definition~\ref{def-P}, we pose an alternative problem statement that removes the constraints to feasible switched system trajectories so that optimization may be performed over the \emph{unconstrained} $(\alpha,\mu) \in \mathcal{X} \times \mathcal{R}$.
\new{
\begin{problem}
\label{prob-unconst}
Find the infimum of the cost $J$ composed with the projection $\mathcal{P}$ with respect to $(\alpha,\mu)\in\mathcal{X}\times\mathcal{R}$:
\[
\inf_{(\alpha,\mu)\in\mathcal{X}\times\mathcal{R}}J(\mathcal{P}(\alpha,\mu)).
\]
\end{problem}
}
Since $\mathcal{P}$ is a projection, Problem~\ref{prob-unconst} is equivalent to Problem~\ref{prob-const}. To see this, note that for any sequence $(\alpha^k,\mu^k)\in\mathcal{X}\times\mathcal{R}$,  $\{J(\mathcal{P}(\alpha^k,\mu^k))\}=\{J(x^k,u^k)\}$ where $(x^k,u^k) = \mathcal{P}(\alpha^k,\mu^k)\in\mathcal{S}$.  Likewise, for any sequence $(x^k,u^k)\in\mathcal{S}$,  $\{J(x^k,u^k)\}$ $= \{J(\mathcal{P}(\alpha^k,\mu^k))\}$ where $(\alpha^k,\mu^k)$ is in the pre-image of $\mathcal{P}(x^k,u^k)$.  Note the pre-image of $\mathcal{P}$ is nonempty because $(x,u) = \mathcal{P}(x,u)$.

In this paper, we solve Problem~\ref{prob-unconst} by generating a sequence $\{x^k,u^k\}\in\mathcal{X}\times\mathcal{R}$ for which $\lim_{k\rightarrow \infty}J(\mathcal{P}(x^k,u^k)) = J^\star$,  where $J^\star$ is a local infimum of Problem~\ref{prob-unconst}. 
Specifically, for each $k$, we compute an update $(x^k,u^k)\rightarrow (x^{k+1},u^{k+1})$ of the form:
\begin{equation}
(x^{k+1},u^{k+1}) = \mathcal{P}(x^k,u^k-\gamma^k d^k)
\label{eq-update}
\end{equation}
starting with $u^0\in\Omega$.  Here, $\gamma^k\in\setR^+$ and $d^k$ is the mode insertion gradient Eq. (\ref{eq-insertion_gradient}) calculated from $u^k$.  From Assumption~\ref{ass-v_crit}, we know that $u^k-\gamma^k d^k\in\mathcal{R}$ and therefore the mapping $\mathcal{Q}:\mathcal{R}\rightarrow\Omega$ is well defined. With this procedure, although each pair $(x^k,u^k-\gamma^k d^k)$ is an element of $\mathcal{X}\times\mathcal{R}$,  $\mathcal{P}(x^k,u^k-\gamma^k d^k)$ is an element of $\mathcal{S}$, and so every iteration is a feasible (non-chattering) switched system trajectory.

The goal is for the sequence of costs $\{J(x^k,u^k)\}$ to converge to a local infimum. Similar to derivative-based iterative optimization (see \cite{armijo,nocedal_wright,kelley}) we need to guarantee a) that a step in the search direction $d^k$ exists that reduces the cost and b) that convergence of $\{x^k,u^k\}$ coincides with $J(x^k,u^k) - J^\star\rightarrow 0$---i.e. that sequence convergence implies that the infimum is found. To guarantee a), we show in Section~\ref{sec-descdir} that the negative mode insertion gradient $-d$ is a descent direction---i.e. that for sufficiently small $\gamma^k$, $J(\mathcal{P}(x^k,u^k-\gamma^k d^k))<J(x^k,u^k)$. As for b), we show in Section~\ref{sec-type2_suffdesc} that there is a connected interval of step sizes which guarantee sufficient descent. Furthermore, in Section~\ref{sec-backtracking}, we provide a means to calculate a step size of sufficient descent using backtracking and provide bounds on the number of backtracking steps required. 

The convergence to an infimum is argued through the optimality conditions from the Hybrid Maximum Principle.

\subsection{Optimality Condition}
Through the hybrid maximum principle \cite{riedinger_kratz_etal,riedinger_iung_etal} expressed for Problem~\ref{prob-const}, we can specify an optimal condition---i.e. an equality $\theta(x^\star,u^\star) = 0$ necessary for the switched system $(x^\star,u^\star)$ to be a solution to Problem~\ref{prob-const} where we define $\theta$ shortly. \new{We use $\theta$ as in \cite{polak_wardi}, as an optimality function}. The condition $\theta=0$ assumes that the cost is so that a feasible (non-chattering) switched system trajectory optimizes the problem. When an optimal switched system trajectory exists, a sequence $\{x^k,u^k\}$ for which $\theta(x^k,u^k)\rightarrow 0$ implies $(x^k,u^k)\rightarrow (x^\star,u^\star)$. When an optimum does not exist, the infimum can only be pursued at the limit as $k\rightarrow \infty$.

The maximum principle expressed for Problem~\ref{prob-const} is as follows, where the Hamiltonian is $H(\rho, \rho_0,x,u,t) := \rho(t)^TF(x(t),u(t)) + p_0\ell(x(t))$ (see \cite{riedinger_kratz_etal} Theorem 1):
\begin{thm}[Switched system maximum principle]
\label{thm-mp}
If $(x^\star,u^\star)\in\mathcal{S}$ is an optimal feasible (non-chattering) switched system trajectory, then there exists an absolutely continuous curve $\rho^\star$ and constant $\rho_0^\star\geq 0$ such that
\begin{enumerate}
\item $\dot{x}^\star(t) = \frac{\partial}{\partial \rho}H(\rho^\star, \rho_0^\star,x^\star,u^\star,t)^T$,
\item $\dot{\rho}^\star(t) = -\frac{\partial}{\partial x}H(\rho^\star, \rho_0^\star,x^\star,u^\star,t)^T$, and
\item $H(\rho^\star, \rho_0^\star,x^\star,u^\star,t) = \min_{\sigma\in\{1,\ldots,N\}}H(\rho^\star, \rho_0^\star,x^\star,e_\sigma,t)$.
\end{enumerate}
\end{thm}
In the maximum principle 1) requires the optimal trajectory must satisfy the state equation Eq.~(\ref{eq-state}) while 2) requires the curve $\rho^\star$ is the solution to the adjoint equation Eq.~(\ref{eq-rhodot}).  Additionally, 3) requires that the Hamiltonian for the optimal mode has least value compared to all other switching controls---recall $e_\sigma$ is the $\sigma^{\textrm{th}}$ column of the $N$ dimensional identity matrix. In general, direct synthesis of the optima is impossible since the three requirements form a boundary value problem which are commonly solved through iterative approaches like the one in this paper.

Requirement 3) can be written in a familiar manner through the mode insertion gradient where, as seen in the following corollary to Theorem~\ref{thm-mp}, the optimality condition is a function of the lower bound on the mode insertion gradient. Define $\theta$ as this lower bound for some switched system $(x,u)\in\mathcal{S}$: \new{First, set $\sigma'\in\{1,\ldots,N\}$ and $T'\in[0,T]$ as the index and timing of the lower bound of $d$:
\begin{equation}
(\sigma', T') = \arg \min_{\sigma\in\{1,\ldots,N\},t\in[0,T]} d_{\sigma}(t).
\label{eq-insertion_mode_time}
\end{equation}
Then,
\begin{equation}
\theta(x,u) := d_{\sigma'}(T')
\label{eq-thetacalc}
\end{equation}
This $\sigma'$ and $T'$ are important as they will be the mode and timing of the initial insertion of $\mathcal{Q}(u-\gamma d)$ as $\gamma$ increases from 0. Additionally, $(\sigma', T')$ may not be unique, but that does not affect the value of $\theta$.
}
\begin{cor}[Optimality condition]
\label{cor-opt_cond}
The switched system $(x^\star,u^\star)$ and cost $J$ with mode insertion gradient $d^\star$ is an optimal feasible (non-chattering) switched system trajectory if $\theta = 0$.
\end{cor}
\begin{proof}
Requirement 3) of Theorem~\ref{thm-mp} is equivalent to: For each $\sigma \in\{1,\ldots,N\}$ and for each $t\in[0,T]$,
\[
H(\rho^\star, \rho_0^\star,x^\star,e_\sigma,t) - H(\rho^\star, \rho_0^\star,x^\star,u^\star,t) \geq 0.
\]
Plugging in for the definition of the Hamiltonian, the left hand side of the inequality is the $\sigma^\textrm{th}$ mode insertion gradient and so for each $\sigma \in\{1,\ldots,N\}$ and for each $t\in[0,T]$,
\[
\rho^\star(t)^T[f_\sigma(x(t)) - F(x(t),u(t))] = d^\star_\sigma(t) \geq 0
\]
Taking the lower bound on $d^\star_\sigma(t)$ for each mode $\sigma$ and time $t$, claim 3) of Theorem~\ref{thm-mp} is equivalent to $\theta= 0$. 
\end{proof}
Through Corollary~\ref{cor-opt_cond}, convergence of a sequence $\{(x^k,u^k)\}$ to an infimizing switched system trajectory is indicated by convergence of the optimality function $\theta(x^k,u^k)$ to 0.

\section{Local Approximation of the Cost}
\label{sec-approx_cost_ST}
The goal of the iterative update Eq.~(\ref{eq-update}) is to generate a sequence of switched systems $\{(x^k,u^k)\}$ with costs $\{J(x^k,u^k)\}$ that converge to a local infimal cost, in order to solve Problem~\ref{prob-unconst} while ensuring that every iterate is in $\mathcal{S}$ (thus enabling applications like receding horizon control, as shown in the examples in Section~\ref{sec-eg}).  In derivative-based optimization the update and convergence guarantees are based on local approximations.  For instance, gradients and Hessians are solutions to local quadratic models \cite{zeidler}. For Problem~\ref{prob-unconst}, the set over which the optimization is occurring  is $\mathcal{X}\times\mathcal{R}$, but $\mathcal{R}$ is not a Hilbert space even when coupled with an inner product. In fact, $\mathcal{R}$ is not a vector space since it does not contain the origin. Even if $\mathcal{R}$ were an inner product space, it would not be complete. Fortunately, though, as we find in this section, the cost can still be approximated in the direction $-d$.

\subsection{Initial Update in the Direction $-d$}
To determine how the iterative update in Eq.~(\ref{eq-update}) varies with $\gamma$, fix $(x,u)\in\mathcal{S}$ and calculate $d$ from Eq.~(\ref{eq-insertion_gradient}). The updated switching control is $\mathcal{Q}(u-\gamma d)$. This section shows that the update is unchanging near $\gamma = 0$.  In other words, there is a $\overline{\gamma}>0$ such that for every $\gamma\in[0,\overline{\gamma})$, $\mathcal{Q}(u-\gamma d) = u$.  In the following lemma we not only show that $\overline{\gamma}$ exists, but calculate the upper bound $\overline{\gamma}$, labelled $\gamma_0$, for a given $d$.  Since $\gamma_0$ is this upper bound, the projection must change for $\gamma$ just right of $\gamma_0$---i.e. $\mathcal{Q}(u-\gamma_0^+ d) \neq u$. 
\new{
Since the projected switching control returns the modes with maximum corresponding value in $u-\gamma d$ for each time, $\gamma_0$ depends on the lower bound of $d$, Eq.~(\ref{eq-insertion_mode_time}), and therefore the value of $\gamma_0$ directly depends on the optimality function $\theta$, Eq.~(\ref{eq-thetacalc}).}
\begin{lem}[$\gamma_0$]\label{lem-gamma0}
For switched system trajectory $(x,u)$ and cost $J$, if $\theta:=\theta(x,u) = 0$, then $\mathcal{Q}(u-\gamma d) = u$ for all $\gamma>0$. Otherwise, the value 
\begin{equation}
\gamma_0 =-\frac{1}{\theta}
\label{eq-gamma0k}
\end{equation}
is such that $\mathcal{Q}(u-\gamma d) = u$ for all $\gamma\in[0,\gamma_0)$ and $\mathcal{Q}(u-\gamma_0^+ d)\neq u$.
\end{lem}

The significance of $\gamma_0$ is that it is the lower bound of $\setR^+$ for which the update $\mathcal{Q}(u-\gamma d)$ is useful. Therefore, we use $\gamma_0$ as the lower bound on the line search in the iterative procedure to solve Problem~\ref{prob-unconst} (see Section~\ref{sec-suffdesc}). Additionally, when an optimum is found---i.e. when $\theta = 0$---Lemma \ref{lem-gamma0} indicates that the projected control is unchanged for any $\gamma>0$.

\subsection{Derivative of the Cost in the Direction $d$ Almost Everywhere}
As $\mathcal{Q}(u-\gamma d)$ varies with $\gamma$, both the switching times $\mathcal{T}$ and the mode sequence $\Sigma$ in the updated mode schedule vary. However, $\Sigma$ will not vary for all $\gamma>\gamma_0$. Define $\Gamma$ as the $\gamma\in\setR^+$ where the mode sequence $\Sigma$ varies:
\begin{equation}
\begin{array}{l}
\Gamma := \{\gamma \in \setR^+| \forall \delta \gamma>0, \exists \gamma'\in(\gamma-\delta\gamma,\gamma+\delta\gamma)\cap\setR^+, \\\hspace{20pt}
\textrm{ where } \Sigma(\mathcal{Q}(u-\gamma d))\neq\Sigma(\mathcal{Q}(u-\gamma' d))\}. 
\end{array}
\label{eq-Gamma}
\end{equation}
For all $\gamma\not\in\Gamma$ only the switching times vary. 
\new{
From Lemma 3 in \cite{caldwell_murphey_cdc12}, due to Assumption~\ref{ass-v_crit}, the cardinality of $\Gamma$ is finite and so $\Gamma$ has zero measure on $\setR^+$. 
}
Define $\Sigma(\gamma) :=\Sigma(\mathcal{Q}(u-\gamma d)) = [\sigma_1,\ldots,\sigma_{M}]$ and $\mathcal{T}(\gamma):=\mathcal{T}(\mathcal{Q}(u-\gamma d)) = [T_1(\gamma),\ldots,T_{M-1}(\gamma)]$ as the updated mode sequence and switching times at $\gamma\not\in\Gamma$. The cost parameterized by the mode schedule is 
\[
J(\Sigma(\gamma),\mathcal{T}(\gamma)):= J(\mathcal{P}(u-\gamma d)).
\]
The derivative of the cost with respect to $\gamma\neq\Gamma$ is
\begin{equation}
\frac{\partial}{\partial \gamma}J(\Sigma(\gamma),\mathcal{T}(\gamma)) =D_2J(\Sigma(\gamma),\mathcal{T}(\gamma)) D\mathcal{T}(\gamma)
\label{eq-DJgamma}
\end{equation}
where $D_2J(\Sigma(\gamma),\mathcal{T}(\gamma)) = \frac{\partial}{\partial \mathcal{T}}J(x,u)$ is the switching time gradient (see Eq.~(\ref{eq-stgrad})). Additionally, $D\mathcal{T}(\gamma)$ is the derivative of the vector of switching times with respect to the step size $\gamma$ and is given in the following lemma, which is Lemma 5 from \cite{caldwell_murphey_cdc12}.  
\begin{lem}[$DT_i(\gamma)$]
\label{lem-DTgamma}
If $\gamma\not\in\Gamma$\textemdash i.e. $\Sigma(\gamma)$ is constant\textemdash then the $i^{\textrm{th}}$ element of the derivative of $\mathcal{T}(\gamma)$, $D\mathcal{T}(\gamma)_i = DT_i(\gamma)$, is given for the following two cases:
\begin{enumerate}
\item If $T_i(\gamma)$ is not a critical time of $\mu_{\sigma_i\sigma_{i+1}}:= u_{\sigma_i\sigma_{i+1}}-\gamma d_{\sigma_i\sigma_{i+1}}$, then
\begin{equation}
DT_i(\gamma) =  -\frac{u_{\sigma_i\sigma_{i+1}}(T_i(\gamma))}{\gamma^2\dot{d}_{\sigma_i\sigma_{i+1}}(T_i(\gamma))}.
\label{eq-Dtau_cont}
\end{equation}
\item If $T_i(\gamma)$ is a discontinuity time of $\mu_{\sigma_i\sigma_{i+1}}$ and 
\[
0\in (\mu_{\sigma_i\sigma_{i+1}}(T_i(\gamma)^-), \mu_{\sigma_i\sigma_{i+1}}(T_i(\gamma)^+)),
\]
then $DT_i(\gamma) = 0$.
\end{enumerate}
\end{lem}
There are times $t$ of $\mu_{\sigma_i\sigma_{i+1}}(\cdot)$ where $t$ is a critical time but not a discontinuity time---e.g. when $\dot{\mu}_{\sigma_i\sigma_{i+1}}(t) = 0$. According to Eq.~(\ref{eq-Dtau_cont}), as a switching time $T_i(\gamma)$ approaches a time $t$ where $\dot{d}_{\sigma_i\sigma_{i+1}}(t) = 0$, $DT_i(\gamma)$ goes unbounded. By Assumption~\ref{ass-v_crit} on $d$, there are only a finite number of critical times. These times are handled in the next section, specifically at step size $\gamma_0$, defined in Eq.~\ref{eq-gamma0k}. 

Eq.~(\ref{eq-DJgamma}), the derivative of the cost with respect to $\gamma$, uses the switching time gradient, Eq.~(\ref{eq-stgrad}), along with the result in Lemma~\ref{lem-DTgamma}. In the next section, we approximate the switching times' dependence on $\gamma$ near $\gamma_0$ uses Lemma~\ref{lem-DTgamma}.

\subsection{Local  Approximation of the Switching Times \label{sec-approx_ST}}
Recall Contribution A in which we wish to locally approximate $J(\mathcal{P}(u-\gamma d))$ in a neighborhood of $\gamma_0$ for $\gamma>\gamma_0$.  Since the size of $\Gamma$ in Eq.~(\ref{eq-Gamma}) is finite, there exists some $\delta\gamma>0$ for which $\Sigma(\gamma)$ is constant for $\gamma\in(\gamma_0,\gamma_0+\delta\gamma)$.  Consequently, only $\mathcal{T}(\gamma)$ varies for $\gamma\in(\gamma_0,\gamma_0+\delta\gamma)$ and the approximation of $J(\mathcal{P}(u-\gamma d))$ in the direction $d$ depends on the approximation of $\mathcal{T}(\gamma)$.

In order for the mode schedule to vary for $\gamma\in(\gamma_0,\gamma_0+\delta\gamma)$ at least one switching time of $\mathcal{T}(\gamma)$ must vary with $\gamma$.  Suppose this switching time is $T_i(\gamma)\in\mathcal{T}(\gamma)$, which separates adjacent modes $\sigma_i$, $\sigma_{i+1}\in\{1,\ldots,N\}$.  We wish to approximate $T_i(\gamma)$ near $\gamma_0$. Since the mode sequence might not be constant at $\gamma_0$---e.g. when $\gamma_0\in\Gamma$---it is possible that switching time $T_i$ exists at $\gamma_0^+$ but not at $\gamma_0^-$. Therefore, we approximate $T_i(\gamma)$ for $\gamma$ in neighborhoods immediately following $\gamma_0$.

Often, a function approximation is made from its Taylor expansion. Here, though, it is not always possible to directly expand $T_i(\gamma)$ around $\gamma_0$ since $DT_i(\gamma)$ can go unbounded when $\gamma$ approaches $\gamma_0^+$.   For example, referring to Eq.~(\ref{eq-Dtau_cont}), $DT_i(\gamma_0^{+})$ is unbounded when $\dot{d}_{\sigma_i\sigma_{i+1}}(T_i(\gamma_0^{+})) = 0$. We find though that even when $DT_i(\gamma_0^+)$ is unbounded, $T_i(\gamma)$ can still be approximated. We label the switching time $T_i(\gamma_0^+)$ with a \emph{type} in order to distinguish between when $\dot{d}_{\sigma_i\sigma_{i+1}}(T_i(\gamma_0^{+}))$ and higher order derivatives are zero or not. 
\begin{definition}\label{def-type}
Suppose $T_i(\gamma_0^+) \in \mathcal{T}(\gamma_0^+)$ is the switching time between modes $\sigma_i$ and $\sigma_{i+1}\in \Sigma(\gamma_0^+)$. The switching time $T_i(\gamma_0^+)$ is of \emph{type} $m(T_i(\gamma_0^{+}))\in\{0,1,\ldots\}$ where $m(T_i(\gamma_0))=0$ if there is $\delta\gamma>0$ such that $T_i(\gamma_0^+)$ is constant for all $\gamma\in(\gamma_0,\gamma_0 + \delta\gamma)$, else
\[
m(T_i(\gamma_0^{+})) :=\min\{m\in\setN | D^md_{\sigma_i\sigma_{i+1}}(T_i(\gamma_0^{+}))\neq 0\}
\]
assuming $d_{\sigma_i\sigma_{i+1}}(T_i(\gamma_0^{+}))$ is as differentiable as needed.
\end{definition}

Figure~\ref{fig-mk1and2} shows two example mode insertion gradients for which type-1 (pictured left) and type-2 (pictured right) switching times occur. The type-0 switching times fall under case 2 of Lemma~\ref{lem-DTgamma} where $DT_i(\gamma_0^+) = 0$.

With just Assumption~\ref{ass-C2bnd}, there is no guarantee that $D^md_{\sigma_i\sigma_{i+1}}(T_i(\gamma_0^{+}))$ exists for $m>2$.  For this reason, we regard the situation when $\dot{d}_{\sigma_i\sigma_{i+1}}(T_i(\gamma_0^{+})) = \ddot{d}_{\sigma_i\sigma_{i+1}}(T_i(\gamma_0^{+})) = 0$ as a degeneracy and we make the additional following assumption
\begin{assumption}\label{ass-type}
Each switching time $T_i(\gamma_0^+)\in\mathcal{T}(\gamma_0^+)$ is of type-0, 1, or 2.
\end{assumption}

\new{
Depending on $f_{\sigma_i}$, $f_{\sigma_{i+1}}$, and $J$, an arbitrary number of derivatives of $d_{\sigma_i\sigma_{i+1}}$ could be zero at a switching time assuming they exist, which affects one's ability to specify a sufficient descent condition and bound the number of line search steps. However, we will find in Corollary~\ref{cor-descent_direction_analytic} in Section~\ref{sec-descdir} that $d$ is a descent direction as long as $\theta<0$ and one of these derivatives is non-zero---in other words, we can still guarantee a reduction to cost is possible. Typically, $\dot{d}_{\sigma_i\sigma_{i+1}}(T_i(\gamma_0^+))\neq 0$ or $\ddot{d}_{\sigma_i\sigma_{i+1}}(T_i(\gamma_0^+))\neq 0$, and when it is we can provide a bound on the number of steps required.

Moreover, recall $d_{ab}(t) = \rho^T(t)(f_a(x(t)) - f_b(x(t)))$, $a\neq b\in\{1,\ldots,N\}$, and that $\dot{d}_{ab}(t) = \dot{\rho}^T(t)(f_a(x(t)) - f_b(x(t))) + \rho^T(t)(\frac{d}{dt}Df_a(x(t)) - \frac{d}{dt}f_b(x(t)))$. The equation for $\ddot{d}_{ab}(t)$ is in the proof of Lemma~\ref{lem-lipschitz} in Eq.~\ref{eq-dab_doubledot}. For $t$ to be a type-2 switching time, $\dot{d}_{ab}(t)$ must be zero which means that $[\rho(t),\dot{\rho}(t)]^T$ is in the null space of $[(f_a(x(t)) - f_b(x(t)))^T, (\frac{d}{dt}Df_a(x(t)) - \frac{d}{dt}f_b(x(t)))^T]$. Certainly, if $f_a(x(t)) - f_b(x(t)) = 0$ and $\frac{d}{dt}Df_a(x(t)) - \frac{d}{dt}f_b(x(t)) = 0$, then $t$ cannot be type-1. The degree of agreement between $f_a$ and $f_b$ can be tested a priori, In general though, the null space condition must be checked at runtime to see if the two vector fields are ``accidentally’’ too similar to each other at a particular state or if their difference happens to lie in the null space. A similar argument can be made for finding $t$ to not be type-2 either, in which case, both $\dot{d}_{ab}(t) = 0$ and $\ddot{d}_{ab}(t) = 0$. In practice, we have only seen type-0, 1, and 2 switching times. There may be a way to show this on full measure sets.
}

\begin{figure}
\centering
\def\svgwidth{240pt}
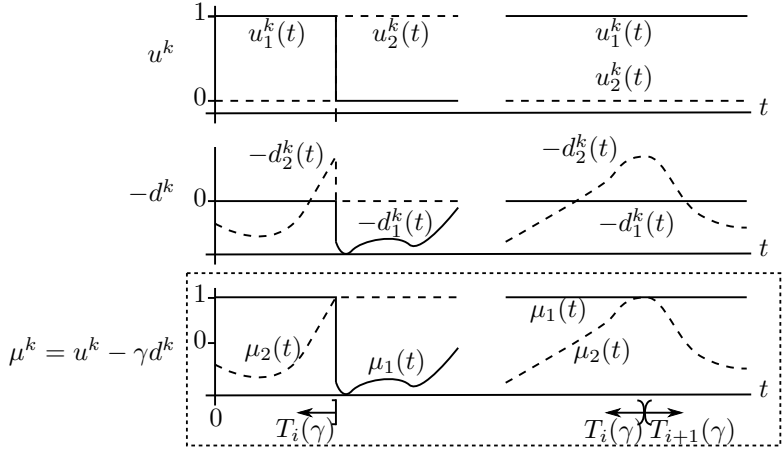
\caption{Example curves $u$, $-d$ and $\mu = u-\gamma_0 d$ showing type-1 (left) and type-2 (right) switching times.  Notice the type-1 switching time is a switching time of $u^k$ and no new insertion occurs, while type-2 switching times come in pairs to insert a mode. The directions in time the switching times vary with $\gamma>\gamma_0$ are also shown.  }
\label{fig-mk1and2}
\end{figure}

Before considering an approximation of $T_i(\gamma)$ for either type-1 or type-2 switching times, we first find that $T_i(\gamma)$ is continuous and strictly monotonic in a neighborhood after $\gamma_0$.

\begin{lem}[Continuity of switching times]
Suppose there exists $\delta\gamma>0$ such that for $\gamma\in(\gamma_0,\gamma_0+\delta\gamma)$, $T_i(\gamma)\in\mathcal{T}(\gamma)$ is the switching time between modes $\sigma_i$ and $\sigma_{i+1}\in\Sigma(\gamma)$.  If $m(T_i(\gamma_0^{+})) = 1$ or $2$, then there is $\overline{\delta\gamma}\in(0,\delta\gamma)$ such that for all $\gamma\in(\gamma_0,\gamma_0+\overline{\delta\gamma})$, $T_i(\gamma)$ is continuous and strictly monotonic.
\label{lem-continuity_of_T}
\end{lem}

The dependence of type-1 and 2 switching times on $\gamma$ have a specific form near $\gamma_0$. The following lemma uses the continuity and strict monotonicity result in Lemma~\ref{lem-continuity_of_T} to specify the dependence of the mode insertion gradient on the switching times. 
\begin{lem}[Dependence of $d(\cdot)$ on $T_i(\gamma)$]
\label{lem-d_properties}
Suppose there exists $\delta\gamma>0$ such that for $\gamma\in(\gamma_0,\gamma_0+\delta\gamma)$, $T_i(\gamma)\in\mathcal{T}(\gamma)$ is the switching time between modes $\sigma_i$ and $\sigma_{i+1}\in\Sigma(\gamma)$ and $T_i(\gamma_0^+)$ is type-1 or 2. Then, there is $\overline{\delta \gamma}\in(0,\delta\gamma]$ such that for all $\gamma\in (\gamma_0,\gamma_0+\overline{\delta \gamma})$, 
\begin{enumerate}
\item $T_i(\gamma)$ is the solution to the following implicit equation:
\begin{equation}
-1-\gamma d_{\sigma_{i+\omega}}(T_i(\gamma)) = 0,
\label{eq-Tigamma_form}
\end{equation}
\item $(-1)^\omega \dot{d}_{\sigma_{i+\omega}}(T_i(\gamma)) >0$, and
\item if $m(T_i(\gamma_0^+)) = 2$, then $\ddot{d}_{\sigma_{i+\omega}}(T_i(\gamma)) >0$,
\end{enumerate}
where $\omega = 0$ if $T_i(\gamma) > T_i(\gamma_0^{+})$ and $\omega = 1$ if $T_i(\gamma) < T_i(\gamma_0^{+})$.
\end{lem}

With Lemma~\ref{lem-continuity_of_T}, which shows that the type-1 and 2 switching times are continuous and strictly monotonic with respect to $\gamma$, and Lemma~\ref{lem-d_properties}, which provides properties of the mode insertion gradient's dependence on the switching times, we can now give approximations of the switching times.  The switching time approximation is used in the next section to approximate the cost function in the direction of the negative mode insertion gradient. The approximation uses the notion of little 'oh', $o(\cdot)$, which is defined as: function $g$ is $o(h)$ if for each $K>0$, there exists a $p_0>0$ such that for all $|p|<p_0$, $|g(p)|<K|h(p)|$. 

The Lemma is as follows:
\begin{lem}[Approximation of switching times]
\label{lem-STapprox}
Suppose there exists $\delta\gamma>0$ such that for $\gamma\in(\gamma_0,\gamma_0+\delta\gamma)$, $T_i(\gamma)\in\mathcal{T}(\gamma)$ is the switching time between modes $\sigma_i$ and $\sigma_{i+1}\in\Sigma(\gamma)$. Then, there is $\overline{\delta \gamma}\in(0,\delta\gamma]$ such that for all $\gamma\in (\gamma_0,\gamma_0+\overline{\delta \gamma})$,  $m(T_i(\gamma_0^{+})) = 1$ implies
\begin{equation}
T_i(\gamma) = T_i(\gamma_0^{+})  - \frac{\theta^2}{\dot{d}_{\sigma_{i+\omega}}(T_i(\gamma_0^{+}))}(\gamma-\gamma_0)+o(\gamma-\gamma_0)
\label{eq-Ti_approx_mk1}
\end{equation}
and $m(T_i(\gamma_0)) = 2$ implies
\begin{equation}
\begin{array}{l}
T_i(\gamma) = T_i(\gamma_0^{+})-\frac{(-1)^\omega \sqrt{2}\theta }{\ddot{d}_{\sigma_{i+\omega}}(T_i(\gamma_0^{+}))^{\frac{1}{2}}}(\gamma-\gamma_0)^{\frac{1}{2}}+o((\gamma-\gamma_0)^{\frac{1}{2}})
\end{array}
\label{eq-Ti_approx_mk2}
\end{equation}
where $\omega = 0$ if $T_i(\gamma) > T_i(\gamma_0^{+})$ and $\omega = 1$ if $T_i(\gamma) < T_i(\gamma_0^{+})$.
\end{lem}

\new{Note that divide by zero is impossible in Eqs.~\ref{eq-Ti_approx_mk1} and \ref{eq-Ti_approx_mk2} without violating the association of the equation with its type.}

\new{
\subsection{Initial Switching Times at $\gamma_0$ \label{sec-switching_time_notation}}
We introduce some notation here in order to keep track of the type of switching times at $\gamma_0^+$. Let $\mathcal{T}(\gamma_0^+)$ be the switching times at $\gamma_0^+$. From Assumption~\ref{ass-type} the greatest type is 2. Partition $\mathcal{T}(\gamma_0^+)$ into sets of equivalent type. Define $I_1$ as the set of indexes of the type-1 switching times at $\gamma_0$ and $I_2$ as the set of indexes of type-2 switching times at $\gamma_0$.  In other words, for $j = 1,2$,
\begin{equation}
I_j = \{i\in\{1,\ldots,M-1\}|m(T_i(\gamma_0^+)) = j\}.
\label{eq-I}
\end{equation}
Further, define 
\begin{equation}
\overline{m}:=max(\{m(T_i(\gamma_0^+))\}_{i=1}^{M-1})
\label{eq-mk}
\end{equation}
to be the greatest type at $\gamma_0^+$. Lemma~\ref{lem-STapprox}  provides the approximation of the switching times for $m(T_i(\gamma_0^+)) = 1$ and $2$.  
}

\subsection{Local Approximation of the Cost \label{sec-approx_cost}}
In smooth finite dimensional optimization, an approximation of the cost in a search direction is the cost's gradient composed with that direction.  We find that the mode insertion gradient, Eq.~(\ref{eq-insertion_gradient}), has a similar role for approximating the projection-based switched system cost. 

The goal is to solve Problem~\ref{prob-unconst} through calculating an infimizing sequence where each iterative update has the form $\mathcal{P}(x,u-\gamma d)$. By approximating the cost as we do in this section, we can specify conditions for which we can guarantee that the sequence's convergence coincides with the infimal cost.  Define $J(\gamma)$ as the change in cost from a fixed $(x,u)\in\mathcal{S}$ in the direction of the negative mode insertion gradient $-d$:
\[
J(\gamma):=J(\mathcal{P}(x,u-\gamma d)).
\]
As indicated in Lemma~\ref{lem-gamma0}, the projected switching control $\mathcal{Q}(u-\gamma d)$ is constant and equal to $u$ until $\gamma>\gamma_0$. Therefore, $J(\gamma) = J(0)$ for $0<\gamma<\gamma_0$. \new{In general, it is possible that $J(\cdot)$ is discontinuous at $\gamma_0$, but when it is continuous $J(\gamma_0^+) = J(0)$. }

\new{
When the switching time of greatest type---i.e. when $\overline{m} = 1$ or $2$---we can guarantee that $J(\gamma)$ is continuous at $\gamma = \gamma_0$.
\begin{lem}
If $\overline{m} = 1$ or $2$, then $J(\gamma)$ is continuous at $\gamma_0$.
\label{lem-Jgamma_continuous}[Continuity of $J$ at $\gamma_0$]
\end{lem}
\begin{proof}
Since the set of $\gamma\in\setR^+$ where the mode schedule of $\mathcal{P}(u-\gamma d)$, $\Gamma$ Eq.~(\ref{eq-Gamma}), varies has measure zero, there is $\delta \gamma > 0$ such that $J(\gamma)$ only depends on $\gamma$ through the switching times $\mathcal{T}(\gamma)$ for all $\gamma_0<\gamma<\gamma_0+\delta \gamma$. Due to Lemma~\ref{lem-continuity_of_T}, there is $\overline{\delta \gamma}$ such that $\mathcal{T}(\gamma)$ is continuous for all $\gamma_0<\gamma<\gamma_0+\overline{\delta \gamma}$
\end{proof}
}

We further analyze $J(\gamma)$ by locally approximating it for $\gamma>\gamma_0$ near $\gamma_0$. Let $\Sigma(\gamma) = [\sigma_1,\ldots,\sigma_M]$ and $\mathcal{T}(\gamma) = [T_1(\gamma),\ldots,T_{M-1}(\gamma)]$ be the mode schedule for $\gamma>\gamma_0$ near $\gamma_0$. The mode sequence $\Sigma(\gamma)$ is constant for some neighborhood greater than $\gamma_0$ since $\Gamma$ has measure zero. Define $\tilde{J}(\gamma)$ as the zeroth-order and first-order terms of the Taylor expansion of $J(\gamma)$, around $\mathcal{T}(\gamma_0^{+})$ so that $J(\gamma) = \tilde{J}(\gamma) + o(\mathcal{T}(\gamma) - \mathcal{T}(\gamma_0^+))$:
\[
\begin{array}{l}
\tilde{J}(\gamma):= J(\gamma_0^{+}) + D_2 J(\Sigma(\gamma_0^{+}),\mathcal{T}(\gamma_0^{+})) (\mathcal{T}(\gamma) - \mathcal{T}(\gamma_0^{+})).
\end{array}
\]
The term $D_2 J(\Sigma(\gamma_0^{+}),\mathcal{T}(\gamma_0^{+}))$ is the switching time gradient Eq.~(\ref{eq-stgrad}) and $\tilde{J}(\gamma)$ becomes
\[
\begin{array}{l}
\tilde{J}(\gamma) = J(\gamma_0^{+}) + \sum_{i = 1}^{M-1} \rho(T_i(\gamma_0^{+}))^T \\\hspace{20pt} \cdot [f_{\sigma_i}(x(T_i(\gamma_0^{+})))- f_{\sigma_{i+1}}(x(T_i(\gamma_0^{+})))](T_i(\gamma) - T_i(\gamma_0^{+})).
\end{array}
\]
There is at least one $T_i(\gamma)\in\mathcal{T}(\gamma)$ that is not constant for $\gamma>\gamma_0$ near $\gamma_0$.  If $T_i(\gamma)$ is increasing in value, the active vector field of $u$ at time $T_i(\gamma)$ is $F(x(t),u(t)) = f_{\sigma_{i+1}}(x(t))$, while, if it is decreasing in value, then $F(x(t),u(t)) = f_{\sigma_{i}}(x(t))$. Assuming $T_i(\gamma)$ increases in value with $\gamma$, the following term is simply the optimality function $\theta$, Eq.~(\ref{eq-thetacalc}):
\[
\begin{array}{l}
\rho(T_i(\gamma_0^{+}))^T[f_{\sigma_i}(x(T_i(\gamma_0^{+}))) - f_{\sigma_{i+1}}(x(T_i(\gamma_0^{+})))] \\\hspace{20pt} = \rho(T_i(\gamma_0^{+}))^T[f_{\sigma_i}(x(T_i(\gamma_0^{+}))) - F(x(T_i(\gamma_0^{+})),u(T_i(\gamma_0^{+})))] \\\hspace{20pt}= d_{\sigma_{i}}(T_i(\gamma_0^{+})) = \theta.
\end{array}
\]
Similarly, assuming $T_i(\gamma)$ decreases in value with $\gamma$, the term is instead $-\theta$:
\[
\begin{array}{l}
\rho(T_i(\gamma_0^{+}))^T[f_{\sigma_i}(x(T_i(\gamma_0^{+}))) - f_{\sigma_{i+1}}(x(T_i(\gamma_0^{+})))] \\\hspace{20pt} = \rho(T_i(\gamma_0^{+}))^T[F(x(T_i(\gamma_0^{+})),u(T_i(\gamma_0^{+}))) - f_{\sigma_{i+1}}(x(T_i(\gamma_0^{+})))] \\\hspace{20pt}= -d_{\sigma_{i+1}}(T_i(\gamma_0^{+})) = -\theta.
\end{array}
\]
Set $\omega_i = 0$ if $T_i(\gamma)$ is increasing in value with $\gamma$ and $\omega_i = 1$ if decreasing\textemdash i.e. $\omega_i = 0$ (alt. $\omega_i = 1$) implies there is $\delta\gamma>0$ such that for each $\gamma\in(\gamma_0,\gamma_0+\delta\gamma)$, $T_i(\gamma) > T_i(\gamma_0^+)$ (alt. $T_i(\gamma)<T_i(\gamma_0^+)$).  Then, $\tilde{J}(\gamma)$ is
\begin{equation}
\tilde{J}(\gamma) = J(\gamma_0^{+}) + \sum_{i = 1}^{M-1} (-1)^{\omega_i}\theta(T_i(\gamma) - T_i(\gamma_0^+)).
\label{eq-Japprox}
\end{equation}
The cost is further approximated by using the switching time approximations in Lemma~\ref{lem-STapprox}. The switching times with the greatest type dominate the approximation of the cost\textemdash e.g. type-1 switching times vary linearly with $\gamma-\gamma_0$ while type-$2$ switching times vary with $(\gamma-\gamma_0)^\frac{1}{2}$.  Label the approximation of the cost with the approximation of the switching times as $\hat{J}(\overline{m};\gamma)$ \new{(Recall the notation in Section~\ref{sec-switching_time_notation} for $I$ and $\overline{m}$)}:
\begin{equation}
\hat{J} (1;\gamma) := J(0) - \sum_{i\in I_1} (-1)^{\omega_i}\frac{(\theta)^3}{\dot{d}_{\sigma_{i+\omega_i}}(T_i(\gamma_0^+))}(\gamma-\gamma_0),
\label{eq-hatJk1}
\end{equation}
and
\begin{equation}
\hat{J} (2;\gamma) := J(0)  - \sum_{i\in I_2} \frac{\sqrt{2}(\theta)^2}{\ddot{d}_{\sigma_{i+\omega_i}}(T_i(\gamma_0^+))^{\frac{1}{2}}}(\gamma-\gamma_0)^{\frac{1}{2}},
\label{eq-hatJk2}
\end{equation}
\new{Note that $J(\gamma_0^+) = J(0)$ due to Lemma~\ref{lem-Jgamma_continuous}.} The following lemma states that $\hat{J}(\overline{m};\gamma)$ dominates the remaining terms of $J(\gamma)$ for $\gamma>\gamma_0$ near $\gamma_0$.  In other words, $\hat{J}(\overline{m};\gamma)$ is a valid local approximation of $J(\gamma)$ near $\gamma_0$.
\begin{lem}[Approximation of the Cost]
\label{lem-Japprox}
Set $J(\gamma) = \hat{J}(\overline{m};\gamma) + R(\gamma)$ where $R(\gamma)$ is the remainder.  If $\overline{m} = 1$ or $2$, then there exists $\delta\gamma>0$ such that for all  $\gamma\in(\gamma_0,\gamma_0+\delta\gamma)$, $|\hat{J}(\overline{m};\gamma)-J(0)| \geq  |R(\gamma)|$.
\end{lem}

Lemma~\ref{lem-Japprox} shows that the approximation of the cost in the direction of the negative mode insertion gradient $\hat{J}(\overline{m};\gamma)$ (Contribution A of the paper) dominates the remaining terms of $J(\gamma)$ in a neighborhood of $\gamma_0$. 
\new{
\subsection{Descent Direction}
\label{sec-descdir}
In order to show sufficient descent (Contribution C) and for backtracking to be applicable (Contribution D), $-d$ must be a descent direction (Contribution B).  In this section we prove $-d$ is a descent direction directly from the approximation of $J(\gamma):=J(\mathcal{P}(u-\gamma d))$ given in Eq.~(\ref{eq-Japprox}).  The search direction $-d$ is defined as a \emph{descent direction} if there is a $\delta\gamma>0$ such that for each $\gamma\in(\gamma_0,\gamma_0+\delta\gamma)$, $J(\gamma)<J(0)$. 

\begin{lem}[Guaranteed descent when $J$ continuous at $\gamma_0$]
If $\theta<0$ and $J(\gamma)$ is continuous at $\gamma_0$, then $-d$ is a descent direction.
\label{lem-descent_direction_J_continuous}
\end{lem}
\begin{proof}
Recall $\tilde{J}(\gamma)$, Eq.~(\ref{eq-Japprox}), is the first and zeroth-order terms of the Taylor Expansion of $J(\gamma)$ around $\mathcal{T}(\gamma_0^+)$. Therefore, 
\[
J(\gamma) = \tilde{J}(\gamma) + o(\mathcal{T}(\gamma) - \mathcal{T}(\gamma_0^+)).
\]
Recalling $\theta<0$ and considering Eq.~\ref{eq-Japprox}, if $T_i(\gamma)$ is increasing with $\gamma$, then $\omega_i = 0$ and $T_i(\gamma) - T_i(\gamma_0^+)>0$, and therefore $\tilde{J}(\gamma)<J(\gamma_0^+)$. Likewise, when $T_i(\gamma)$ is decreasing, we can conclude $\tilde{J}(\gamma)<J(\gamma_0^+)$. Thus, $\tilde{J}(\gamma)<J(\gamma_0^+)$. Since $\tilde{J}(\gamma)$ dominates $o(\mathcal{T}(\gamma) - \mathcal{T}(\gamma_0^+)$ near $\mathcal{T}(\gamma_0^+)$, there is a $\delta \gamma>0$ such that $J(\gamma) < J(\gamma_0^+)$ when $\gamma\in(\gamma_0 , \gamma_0 + \delta \gamma)$. Finally, since $J(\gamma)$ is continuous at $\gamma = \gamma_0$, $J(\gamma_0^+) = J(0)$. 
\end{proof}

With Lemma~\ref{lem-descent_direction_J_continuous}, it directly follows that $-d$ is guaranteed to be a descent direction if $\overline{m} = 1$ or $2$, or if $f_\sigma$, $\sigma\in\{1,\ldots,N\}$, and $\ell$ are analytic in $x$. The second follows from the absolute continuity of $J(\gamma)$ result in \cite{caldwell_murphey_cdc12} and allows for higher type switching times than 0, 1, or 2.
\begin{cor}[Descent Direction]
If $\theta<0$, and $\overline{m} = 1$ or $2$, then $-d$ is a descent direction.
\label{cor-descent_direction}
\end{cor}
\begin{proof}
As indicated in Lemma~\ref{lem-Jgamma_continuous}, $J(\gamma)$ is continuous at $\gamma = \gamma_0$ when $\overline{m} = 1$ or $2$. Descent is thus guaranteed through Lemma~\ref{lem-descent_direction_J_continuous}.
\end{proof}

\begin{cor}[Descent Direction]
If $\theta<0$, and $f_\sigma$, $\sigma\in\{1,\ldots,N\}$ and $\ell$ are analytic in $x$, then $-d$ is a descent direction.
\label{cor-descent_direction_analytic}
\end{cor}
\begin{proof}
As indicated in Theorem 5.2 in \cite{caldwell_murphey_cdc12}, $J(\gamma)$ is absolutely continuous (and therefore continuous) at $\gamma = \gamma_0$ assuming each mode $f_i$ and $\ell$ are analytic.  Descent is thus guaranteed through Lemma~\ref{lem-descent_direction_J_continuous}.
\end{proof}
}

Notice that Corollary~\ref{cor-descent_direction} does not need to strengthen Assumption~\ref{ass-C2bnd} to include analyticity requirements on the objective and dynamics, while Corollary~\ref{cor-descent_direction_analytic} does not need the type-0, 1, and 2 switching time assumption, Assumption~\ref{ass-type}.

\section{Sufficient Descent}
\label{sec-suffdesc}
Consider the following iterative algorithm:
\begin{framed}
\begin{alg}\label{alg-iter_update}
With $(x^0,u^0)\in\mathcal{S}$, execute
\[
(x^{k+1},u^{k+1}) = \mathcal{P}(x^k,u^k-\gamma^kd^k)
\]
where for each $k = 0,1,2,\ldots$, $\gamma^k>\gamma_0^k$.
\end{alg}
\end{framed}
For the remainder of the paper, the superscript $k$ signifies that the corresponding variable or mapping depends on $(x^k,u^k)$. For instance, $d^k:=d(x^k,u^k)$, $\theta^k:=\theta(x^k,u^k)$, and $J^k(\gamma)$ is $J(\mathcal{P}(x^k,u^k-\gamma d^k))$.

Algorithm~\ref{alg-iter_update} corresponds to repeatedly stepping in the direction given by the negative mode insertion gradient and projecting to a feasible switched system trajectory. The algorithm's desired result is to generate a sequence that converges to a local infimal cost in order to solve Problem~\ref{prob-unconst}. Through the descent direction result in Corollary~\ref{cor-descent_direction}, there always exists a $\gamma^k$ such that $J(\mathcal{P}(x^k,u^k-\gamma^kd^k))<J(x^k,u^k)$ as long as $\theta^k<0$ and through Corollary~\ref{cor-opt_cond}, if $\theta^k = 0$, then $(x^k,u^k)$ is optimal. By choosing a $\gamma^k$ that reduces the cost at each iteration of Algorithm~\ref{alg-iter_update}, the resulting sequence $\{x^k,u^k\}$ is such that $\{J(x^k,u^k)\}$ is strictly monotonically decreasing. As such, assuming $J(\cdot)$ is bounded below by $\underline{J}\in\setR$, then the sequence $\{J(x^k,u^k)\}$ is guaranteed to converge. However, there is as of yet no guarantee that $\{J(x^k,u^k)\}$ converges to an infimum. This section provides a means to calculate $\gamma_k$ to guarantee Algorithm~\ref{alg-iter_update} converges to a local infimal cost.  The convergence is proven by showing that the sequence of optimality functions goes to zero at the limit---i.e. $\theta^k\rightarrow 0$.

In this section, we give the sufficient descent condition (Contribution C), show that a step size $\gamma^k$ that satisfies the sufficient descent condition can be calculated in a finite number of backtracking iterations (Contribution D) and finally that executing Algorithm~\ref{alg-iter_update} for such a $\gamma^k$ results in $\lim_{k\rightarrow \infty}\theta^k = 0$.  Each of these contributions follow from the approximation of the cost (Contribution A).  

\subsection{Type-2 Sufficient Descent Condition \label{sec-type2_suffdesc}}
The sufficient descent condition (Contribution C) follows directly from the approximation of the cost $\hat{J}^k(\overline{m}^k;\gamma)$, Eqs.~(\ref{eq-hatJk1}) and (\ref{eq-hatJk2}) (Contribution A).  Set $\alpha \in (0,1)$.  The type-$m^{k}$ sufficient descent condition is 
\[
J^k(\gamma) - J^k(0) < \alpha (\hat{J}^k(\overline{m}^k;\gamma) - J^k(0)).
\]
The condition is an upper bound on the reduction of cost between successive iterations as a function of $\gamma-\gamma_0^k$.  Since this bound is the scaled approximation of the cost, there are $\gamma$ near $\gamma_0^k$ that will satisfy the inequality.  Lemma~\ref{lem-suff_dec}, presented shortly, provides an interval of such $\gamma$. In Section~\ref{sec-min_seq} we show that a sequence generated by Algorithm~\ref{alg-iter_update} converges with a properly chosen $\gamma^k$ by showing that $\theta^k$ goes to zero.

\new{
The type-1 and 2 sufficient descent conditions are:
\begin{definition}\label{def-suffdesc_type1}
Set 
\begin{equation}
s^k_1 =  \sum_{i\in I^k_1} (-1)^{\omega_i}\frac{(\theta^k)^3}{\dot{d^k}_{\sigma_{i+\omega_i}}(T_i(\gamma^{k^+}_0))}.
\label{eq-sk1}
\end{equation}
The \emph{type-1 sufficient descent} condition is 
\begin{equation}
\begin{array}{l}
J^k(\gamma) - J^k(0) < \alpha s^k_1 (\gamma - \gamma_0^k).
\end{array}
\label{eq-t1suff_dec}
\end{equation}
\end{definition}

\begin{definition}\label{def-suffdesc}
Set 
\begin{equation}
s^k_2 =  - \sum_{i\in I^k_2} \frac{\sqrt{2}(\theta^k)^2}{\ddot{d}^{k}_{\sigma_{i+\omega_i}}(T_i(\gamma_0^{k^+}))^{\frac{1}{2}}}.
\label{eq-sk2}
\end{equation}
The \emph{type-2 sufficient descent} condition is 
\begin{equation}
\begin{array}{l}
J^k(\gamma) - J^k(0) < \alpha s^k_2 (\gamma - \gamma_0^k)^\frac{1}{2}.
\end{array}
\label{eq-t2suff_dec}
\end{equation}
\end{definition}
}

We study the type-2 sufficient descent condition---i.e. when $\overline{m}^k = 2$.  For $\overline{m}^k = 1$, type-1 switching times occur at switching times of $u^k$ or at the boundary times.  Since the type-1 switching time approximation is linear in $(\gamma - \gamma_0^k)$, sufficient descent and backtracking directly correspond to switching time optimization\textemdash see \cite{caldwell_murphey_journal,egerstedt_wardi_axelsson,flasskamp_murphey_oberblobaum,xu_antsaklis_2} for switching time optimization.  

The following Lemma shows that there exists a $\hat{\gamma}>\gamma_0^k$ for which each $\gamma\in(\gamma_0^k,\hat{\gamma}]$ satisfies the type-2 sufficient descent condition.  As given in the lemma, the step size $\hat{\gamma}$ is the minimum of $\gamma_1^k$, $\gamma_2^k$ and $\gamma_3^k$.  The first step size $\gamma_1^k$ is the largest $\gamma_1^k>\gamma_0^k$ such that for each $\gamma\in(\gamma_0^k,\gamma_1^k)$, $J^k(\gamma)$ is differentiable.  The second, $\gamma_2^k$, depends on the greatest Lipschitz constant $L$ that satisfies the Lipschitz condition on the second time derivative of $d^k$ in Lemma~\ref{lem-lipschitz}.  The third, $\gamma_3^k$, is a constant scaling from $\gamma_0^k$\textemdash i.e. $\gamma_3^k = \gamma_0^k\kappa$ where depending on $\alpha\in(0,1)$, $\kappa$ is between $2-(\sqrt[3]{\frac{3\sqrt{2}}{2}})/3\approx 1.5717$ and $2$.  In the following Lemma, set $\nu := \min_{i\in I^k_2} \ddot{d}^k_{\sigma_{i+\omega_i}}(T_i(\gamma_0^{k^+}))$, where $I^k_2$ is the set of type-2 switching time indices and is defined in Eq.~(\ref{eq-I}).

\begin{lem}[Sufficient descent]
\label{lem-suff_dec}
Suppose $\overline{m}^k = 2$, $\theta^k<0$, and $\gamma_1^k>\gamma_0^k$ is such that for each $\gamma\in(\gamma_0^k,\gamma_1^k)$, $DJ(\gamma)$ exists.  Set
\[
\begin{array}{ccc}
\gamma^k_2 := \gamma_0^k\left(1 - \frac{\nu^3}{\theta^k16L^2}\right) & \textrm{and} &\gamma_3^k :=  \gamma_0^k\left(2-\frac{\sqrt[3]{\alpha \frac{3\sqrt{2}}{2}}}{3}\right).
\end{array}
\]
Defining $\hat{\gamma}^k := \min\{\gamma_1^k,\gamma_2^k,\gamma_3^k\}$, if $\gamma\in(\gamma_0^k,\hat{\gamma}^k]$, then 
\[
J^k(\gamma) - J^k(0) < \alpha s^k_2 (\gamma - \gamma_0^k)^\frac{1}{2}.
\]
.
\end{lem}

\new{
Lemma~\ref{lem-suff_dec} guarantees there will be a step size that satisfies the type-2 sufficient descent condition.  In practice, $\gamma_2^k$ cannot be calculated directly because the Lipschitz constant $L$ is unknown and so backtracking is used instead to find a step size that satisfies sufficient descent. 
}

\subsection{Backtracking \label{sec-backtracking}}
Calculating $\hat{\gamma}^k=\min\{\gamma_1^k,\gamma_2^k,\gamma_3^k\}$ directly is computationally inefficient due to $\gamma_2^k$.  Calculating $\gamma_1^k$ and $\gamma_3^k$ is possible though: $\gamma_1^k$ is the nearest  $\gamma>\gamma_0^k$ to $\gamma_0^k$ for which $J^k(\gamma)$ is not differentiable and therefore, $\gamma_1^k$ is calculated from knowledge of the critical times of $u^k$ and $d^k$; $\gamma_3^k$ is a constant scaling from $\gamma_0^k$.  Calculating $\gamma^k_2$ requires a priori knowledge of the Lipschitz constant $L$. Similar to smooth finite dimensional optimization \cite{armijo,lemarechal}, it is more efficient to calculate a step size that satisfies the sufficient descent criteria using a backtracking method than it is to calculate $\gamma_2^k$ and thus $\hat{\gamma}^k$ directly. We wish to sample $(\gamma_0^k,\gamma_3^k)$ to find a $\gamma$ that satisfies sufficient descent: set $\gamma^k(j) := (\gamma_3^k-\gamma_0^k)\beta^j+\gamma_0^k$ where $\beta\in(0,1)$ and define $j^k\in\{0,1,\ldots\}$ as
\begin{equation}
j^k := \min\{j = 0,1,\ldots| J^k(\gamma^k(j)) - J^k(0)| < \alpha s^k_2 (\gamma^k(j)-\gamma_0^k)^\frac{1}{2}\}.
\label{eq-jk}
\end{equation}
Then, $\gamma^k:= \gamma^k(j^k)$ satisfies the sufficient descent condition.  

\new{
The following algorithm calculates $\gamma^k$ using backtracking for $\overline{m}^k =1$ or $2$ even though we only analyze $\overline{m}^k =2$.  It should be implemented as an inner loop of Algorithm~\ref{alg-iter_update}.

\begin{framed}
\begin{alg}\label{alg-backtracking}%
Set $j^k = 0$ and calculate $s^k_{\overline{m}^k}$ from Eq.~(\ref{eq-sk2}).  
\begin{enumerate}
\item If $J^k(\gamma^k(j^k)) - J^k(0) < \alpha s^k_{\overline{m}^k} (\gamma^k(j^k)-\gamma_0^k)^\frac{1}{\overline{m}^k}$ then return $\gamma^k = \gamma^k(j^k)$ and terminate.
\item Increment $j^k$ and repeat from Step 1.
\end{enumerate}
\end{alg}
\end{framed}
}

The number of backtracking steps is finite:

\begin{lem}[Backtracking]
If $\overline{m}^k = 2$ and there exists $b_1>0$ and $b_2>0$ such that $\theta^k<-b_1$ and for each of the $i\in I^k_2$, $\ddot{d}^k_{\sigma_{i+\omega_i}}(T_i(\gamma))>b_2$, then $j^k$ is finite.
\label{lem-backstep}
\end{lem}
\begin{proof}
The proof follows from Lemmas~\ref{lem-continuity_of_T} and \ref{lem-suff_dec}.  According to Lemma~\ref{lem-suff_dec}, $\hat{\gamma}^k = \min \{\gamma_1^k,\gamma_2^k,\gamma_3^k\}$ satisfies the sufficient descent condition.  From Lemma~\ref{lem-continuity_of_T}, $\gamma_1^k$ is bounded from $\gamma_0^k$.  Furthermore, by the bounds on $\theta^k$ and  $\ddot{d}^k_{\sigma_{i+\omega_i}}(T_i(\gamma))$, $\gamma_2^k$ and $\gamma_3^k$ are bounded from $\gamma_0^k$.  Let $b_3>0$ be this minimal bound of $\hat{\gamma}^k = \min\{\gamma_1^k,\gamma_2^k,\gamma_3^k\}$ from $\gamma_0^k$.  Then, 
\[
j^k = \textrm{ceil}\left(\log_\beta\frac{b_3}{\gamma_3^k - \gamma_0^k}\right)
\]
which is finite, where the function $\textrm{ceil}:\setR\rightarrow\setZ$ rounds to the nearest integer of greater value.
\end{proof}

\subsection{Locally Infimizing Sequence \label{sec-min_seq}}
For the type-2 sufficient descent condition, we have shown backtracking will find a $\gamma^k$ for which the condition is satisfied.  In the following lemma, we find that if $\{x^k,u^k\}$ is the sequence calculated from Algorithm~\ref{alg-iter_update} initialized with $(x^0,u^0)\in\mathcal{S}$ where there is an infinite subsequence of $\{x^k,u^k\}$ for which $\overline{m}^k = 2$, then the optimality function $\theta^k$ goes to zero.  

\begin{lem}[Infimizing Sequence]
\label{lem-min_seq}
Suppose $(x^0,u^0)\in\mathcal{S}$ and $S = \{x^k,u^k\}$ is an infinite sequence where
\begin{enumerate}
\item $J(x^0,u^0)=\overline{J}<\infty$,
\item $J(x,u)$ is bounded below for all $(x,u)\in\mathcal{S}$,
\item $J(x^{k+1},u^{k+1})<J(x^k,u^{k})$, and
\item $S_2\subset S$ is an infinite subsequence where each $(x^{k+1},u^{k+1})\in S_2$ is $\\(x^{k+1},u^{k+1}) = \mathcal{P}(x^k,u^{k}-\gamma^{k} d^{k})$ and
\begin{enumerate}
\item $\overline{m}^k = 2$ (see Eq.~(\ref{eq-mk})),
\item $\gamma_2^{k}<\gamma_1^{k}$ or $\gamma_3^k<\gamma_1^{k}$ (see Lemma~\ref{lem-suff_dec}),
\item there is $K_2>0$ such that for each $i\in I^k_2$, $\ddot{d}^{k}_{\sigma_{i+\omega_i}}(T_i(\gamma_0^{k}))\geq K_2|\theta^k|$, and
\item $\gamma^{k} = (\gamma_3^{k}-\gamma_0^{k})\beta^{j^{k}}+\gamma_0^{k}$ (see Eq.~(\ref{eq-jk})).
\end{enumerate}
\end{enumerate}
then, $\lim_{k\rightarrow \infty} \theta^k = 0$.
\end{lem}

Lemma~\ref{lem-min_seq} provides conditions for which a sequence of switched system trajectories $\{x^k,u^k\}$ are guaranteed to be an infimizing sequence through the guarantee that $\theta^k\rightarrow 0$. Such infimizing sequences can be computed through the iterative update Algorithm~\ref{alg-iter_update} with an inner loop of Algorithm~\ref{alg-backtracking} for type-2 sufficient descent through backtracking. 

Lemma~\ref{lem-min_seq}, states that if the sequence of costs $\{J(x^k,u^k)\}$ is monotonically decreasing and an infinite subset of the iterative updates in Algorithm~\ref{alg-iter_update} satisfy assumptions 4a-4d in the lemma, then the algorithm converges to an infimum. The restrictive assumptions are 4b and 4c. As for assumption 4a and 4d: assumption 4a requires that the greatest switching time type is 2 while assumption 4d requires that the step size is computed through backtracking, Algorithm~\ref{alg-backtracking}. 

For assumption 4b, recall $\gamma_1^k$ is the maximum step size for which $J^k(\gamma)$ is differentiable for all $\gamma\in (\gamma_0^k,\gamma_1^k)$. Since the approximation used by the steepest descent condition is only valid for intervals of $\gamma$ where $J^k(\gamma)$ is differentiable, Lemma~\ref{lem-min_seq} cannot guarantee $\theta^k\rightarrow 0$ if $\gamma_1^k-\gamma_0^k$ goes to $0$ faster than $\theta^k$, which can happen when multiple mode insertions occur near the same time. Violations to assumption 4b can be assessed through comparing the sequence $\{\gamma_1^k-\gamma_0^k\}$ to $\{\theta^k\}$ during execution.

Like with assumption 4b, violations to assumption 4c can be checked as part of an iterative algorithm by comparing the sequence $\{\ddot{d}^{k}_{\sigma_{i+\omega_i}}(T_i(\gamma_0^{k}))\}$ to $\{\theta^k\}$. When a violation to either assumption 4b or 4c occurs, a number of strategies are viable to correct the violation while maintaining the convergence guarantee of Lemma~\ref{lem-min_seq} which need only maintain the decreasing monotonicity of $\{J(x^k,u^k)\}$. Such strategies could execute a step of switching time optimization, mask specific time intervals of the mode insertion gradient, or employ a sign preserving transformation to the mode insertion gradient.  Analyzing such strategies is future work.

It is worth noting that violating assumptions 4b and 4c do not affect the optimality condition, only the convergence of the iterative algorithm.

\new{
\section{Mode Scheduling Algorithm}
\label{sec-mode_sched}
The complete mode scheduling algorithm is as follows. 

\begin{framed}
\begin{alg}\label{alg-mode_sched}
With $(x^0,u^0)\in\mathcal{S}$, $\theta_{stop}<0$, set $k=0$ and execute

While $\theta^k < \theta_{stop}$
\begin{enumerate}
\item \hspace{4pt} Compute $\rho$, i.e. solve Eq.~(\ref{eq-rhodot})
\item \hspace{4pt} $d^k_{\sigma} = \rho^T(f_\sigma(x^k) - F(x^k,u^k))$, $\sigma\in\{1,\ldots,N\}$
\item \hspace{4pt} $\theta^k = \min_{\sigma\in\{1,\ldots,N\}, t\in[0,T]} d^k_\sigma(t)$
\item \hspace{4pt} $\gamma_0^k = -1/\theta^k$
\item \hspace{4pt} Compute $\mathcal{T}(\gamma_0^+)$ 
\item \hspace{4pt} Compute $\overline{m}^k$, Eq.~(\ref{eq-mk})  
\item \hspace{4pt} If $\overline{m}^k \neq 1$ or $2$, Return failure
\item \hspace{4pt} Compute $I_{\overline{m}^k}^k$, Eq.~(\ref{eq-mk})  
\item \hspace{4pt} Compute $s^k_{\overline{m}^k}$, Eq.~(\ref{eq-sk1}) or \ref{eq-sk1}
\item \hspace{4pt} Compute $\gamma^k$, i.e. backtrack using Algorithm~\ref{alg-backtracking}
\item \hspace{4pt} $(x^{k+1},u^{k+1}) = \mathcal{P}(x^k,u^k-\gamma^kd^k)$
\item \hspace{4pt} $(\Sigma^{k+1},\mathcal{T}^{k+1}) = u^{k+1}$
\item \hspace{4pt} Check for convergence failure, 4b and 4c in Lemma~\ref{lem-min_seq}
\item \hspace{4pt} k = k+1
\end{enumerate}
\end{alg}
\end{framed}

The following remarks are with respect to implementing the mode scheduling algorithm:
\begin{enumerate}
\item For step 3, $\theta^k$ depends on the infimizing time of the mode insertion gradient. Additionally, for step 11, $\mathcal{P}$ depends on the times that the curve of $u^k-\gamma^kd^k$ with maximal value changes. For implementation, both operations call for finding these critical times on curves. To find them, we partition $[0,T]$ into subintervals and determine which intervals contain a critical time. Then, we find a critical time on an interval through bisection. In order to not miss a critical time, the partition must be sufficiently fine to ensure that each subinterval contains at most one. Note that the critical times could be the switching times of $\mathcal{T}^k$.
\item The initial switching times $\mathcal{T}(\gamma_0^+)$ in step 5 is the union of the previous switching times $\mathcal{T}^k$ with the times that minimize the mode insertion gradient in step 3. Let $(\sigma', T')$ be the minimizers of step 3. If $\dot{d}_{\sigma'}(T') = 0$, then $\mathcal{T}(\gamma_0^+)$ has two new switching times at $T'$, one increasing in value with $\gamma$ and one decreasing. If $\ddot{d}_{\sigma'}(T') \neq 0$, then these switching times are type-2. If $T'$ is a discontinuity time of $d_{\sigma'}$, then $T'$ must coincide with a switching time of $\mathcal{T}^k$ since $\ddot{d}^k$ is Lipschitz for time intervals between adjacent switching times. In this case, $\mathcal{T}(\gamma_0^+)$ has a single new switching time at $T'$, where $T'$ increases or decreases with $\gamma$ depending on whether $d_{\sigma'}(T'^-)<d_{\sigma'}(T'^+)$ or not.
\item Step 13 checks that the assumptions 4b and 4c in Lemma~\ref{lem-min_seq}, as discussed at the end of Section~\ref{sec-min_seq} are not violated.
\item Depending on how one represents $u^k$ and $(\Sigma^k,\mathcal{T}^k)$ in implementation, step 12 may not be needed.
\end{enumerate}
}

\section{Example}
\label{sec-eg}
\new{
We provide two examples. The first is concerned with a simple vehicle tracking a desired maneuver where the operational modes differ in vehicle velocity and steering angle.  The second example is concerned with responding to a multimachine power network disturbance through controlling a hybrid component that switches the network admittance. For both examples we apply mode scheduling in Algorithm~\ref{alg-mode_sched}.

\subsection{Vehicle Tracking}
As a simple example, we consider a vehicle moving in the plane that can switch between four modes of operation. The state of the system is $x = [X,Y,\psi]^T$ where $(X,Y)$ is the vehicle position and $\psi$ is its orientation. The $\sigma^\textrm{th}$, $\sigma\in\{1,\ldots,4\}$ mode of operation is:
\[
\dot{x}(t) = f_\sigma(x(t))=\left[\begin{array}{l}\dot{X} \\ \dot{Y} \\ \dot{\psi}\end{array}\right](t) = \left[\begin{array}{c} v_\sigma(t)\cos{\psi(t)} \\ v_\sigma(t)\sin{\psi(t)} \\ \omega_\sigma(t) \end{array}\right]
\]
where $v_\sigma$ is the mode's speed and $\omega_\sigma$ is its steering angle with values given in Table \ref{tab-veh_modes}.

\begin{table}[h!]
\centering
\begin{tabular}{c|c|c}
$\sigma$ & $v_\sigma$ & $\omega_\sigma$ \\
\hline
1 & 4.5 & $\pi/3$ \\
2 & 4.5 & $-\pi/3$ \\
3 & 2 & $\pi/3$ \\
4 & 2 & $-\pi/3$.
\end{tabular}
\caption{Velocities and steering angles of the vehicle's modes of operation.}
\label{tab-veh_modes}
\end{table}
That is, the vehicle can turn left or right at two different speeds. The goal is to track the following desired trajectory from initial configuration $x_0 = [0,0,0]^T$:
\[
x_d(t) = \left[
\begin{array}{c}
X_d(t)\\
Y_d(t)\\
\psi_d(t)
\end{array}
\right]
= \left[
\begin{array}{c}
6.5 - 4\cos(t) \\
-1.5 + 4\sin(t) \\
-t + \pi/2.
\end{array}
\right].
\]
The desired trajectory travels a circle of radius $4$ with and center $[6.5,-1.5]^T$ at 1 rad/s counterclockwise starting at $x_d = [2.5,-1.5,\pi/2]^T$. The cost is set as the error from desired trajectory---i.e. $\ell(x(t),u(t)) = 1/2 (x(t)-x_d(t))^T(x(t)-x_d(t))$. 

We successfully execute Algorithm~\ref{alg-mode_sched} with backtracking parameters $\alpha = 0.4$ and $\beta = 0.4$ for 50 iterations.  The results are shown in Figure~\ref{fig-vehicle_results}. The cost drops from $J(x^0,u^0) = 276.37$ to $J(x^{50},u^{50}) = 1.30$ with a significant initial reduction in the first 6 iterations to $J(x^6,u^6) = 1.58$. The optimality function increases from $\theta^0 = -588.67$ to $\theta^{50} = -0.81$. Again, there is a significant change in the first 6 iterations to $\theta^6 = -2.93$.  In Figure~\ref{fig-vehicle_results}d, the $(X,Y)$ plots of the trajectory show that after an initial transition period to the desired trajectory, the vehicle decently tracks the trajectory for the remainder of the time horizon, even by the $6^{\textrm{th}}$ iteration. 

\begin{figure*}
\centering
\def\svgwidth{0.95\textwidth}
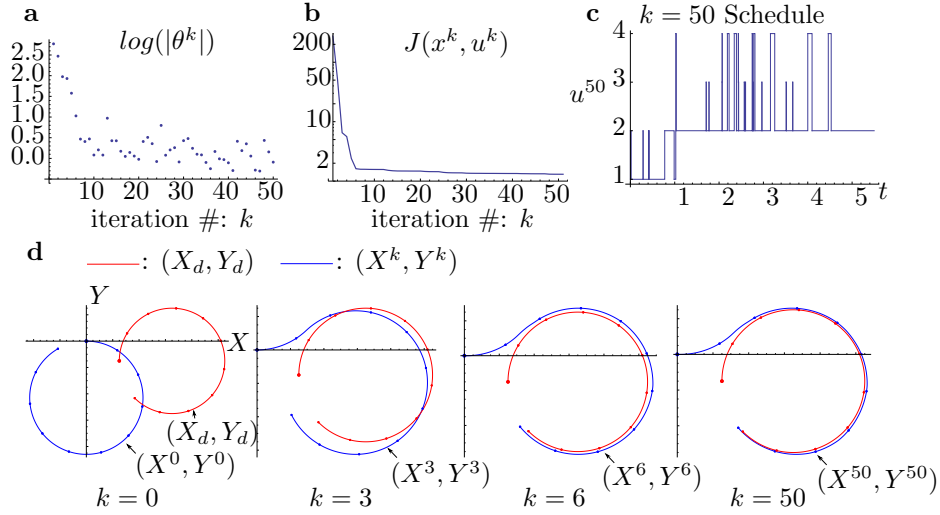
\caption{\textbf{a}) Convergence of optimality function toward zero as a function of iteration.  \textbf{b}) Cost $J$ as a function of iteration.  A large reduction occurs in the first 6 iterations. \textbf{c}) Switching control illustrating the schedule of the 50th iteration. \textbf{d}) Comparison of the $(X,Y)$ trajectories per iteration.}
\label{fig-vehicle_results}
\end{figure*}
}

\subsection{Multimachine Power Network}
Due to the complex interconnectedness of multimachine power networks, it is unclear how to actively reject a disturbance.  The solution we propose is to compute a schedule for physical switches that connect and disconnect capacitors from the network so that system performance improves.  A power network is often modeled as a synchronous machine where the dynamics are given by the swing equations \cite{grainger_stevenson}.  The swing equations are second-order nonlinear differential equations which dictate the evolution of each generator's rotor angle.  The rotors are assumed to be spinning at a constant frequency\textemdash e.g. 60 Hz\textemdash but each rotor's relative phase may not be constant. This is a standard assumption.  The evolution of a single rotor is dictated by the difference of its relative phase with its neighboring generators as well as the admittance of the adjacent power lines and buses. Through switching capacitors, the power lines' admittance value switches, effectively splitting the system dynamics into distinct operating modes dependent on the position of the switches. The only control authority we impose is through the switches.

\new{
The dynamics for a single mode according to the swing equations are as follows: Let $\delta_i(t)$ be the rotor phase of generator $i$ relative to a reference generator, generator 0.  The evolution of the $i^\textrm{th}$ rotor is dictated by the difference of the mechanical power input with the electrical power output:
\begin{equation}
\frac{2H_i}{\omega_s}\ddot{\delta}_i = P_{m,i} - P_{e,i}
\label{eq-pn_dynamics}
\end{equation}
where $\omega_s$, in $rads/s$ is the synchronous speed and $H_i$ is a normalized inertial constant so that the mechanical power $P_{m,i}$ and the electrical power $P_{e,i}$ are in per unit.  The terms $\omega_s$, $H_i$ and $P_{m,i}$ are assumed to be constant for the short time horizon for which the disturbance and resolution occurs. The electrical power output of generator $i$, $P_{e,i}$, depends on the difference of its rotor's relative phase with the neighboring generators' as well as the admittance of the adjacent lines and buses:
\begin{equation}
P_{e,i} = |E_i|^2G_{ii}+\sum_{j\neq i}|E_i||E_j||Y_{ij}|\cos(\delta_i-\delta_j-\psi_{ij})
\label{eq-powerout}
\end{equation}
where $E_i$ is the transient internal voltage, $G_{ii}$ is the real part of the $ii^\textrm{th}$ component of the bus admittance matrix, $Y_{ij}$ is the $ij^\textrm{th}$ component of the bus admittance matrix, and $\psi_{ij}$ is the angle of the $ij^\textrm{th}$ component of the bus admittance matrix.  
}

The example power network has topology and line and bus parameters from the IEEE 118 Bus Test Case, a 1962 study of a segment of North America's midwest grid \cite{christie}. This network is composed of 118 buses, 186 lines, 54 generators and is shown in Fig.~\ref{fig-case118_Graph}.  In addition, we chose 26 power lines to connect switched capacitor banks in series to. 
\new{Following the line numbering form the Test Case, the 26 chosen lines are 6, 9, 14, 29, 38, 39, 43, 49, 57, 59, 77, 85, 92, 100, 113, 120, 126, 129, 134, 140, 141, 153, 165, 172, 176, and 177. }
Each capacitor's capacitance is chosen so that when a switched capacitor is, ``on'', its associated line's reactance doubles.  The location of each capacitor bank is also shown in Fig.~\ref{fig-case118_Graph} and are chosen so that each generator is connected to at least one other generator for which the admittance between the two can be switched.  For this study, all 26 switches are synchronized so that all are ``on'' or ``off'' together. As such, the network has two modes of operation, $f_1$ and $f_2$---i.e. $N = 2$
\new{
---where $f_1$ and $f_2$ are given by Eq.~(\ref{eq-pn_dynamics}) with differing bus admittance matrix $Y$ in Eq.~(\ref{eq-powerout}).
}

\begin{figure*}
\centering
\includegraphics[width = .9\textwidth]{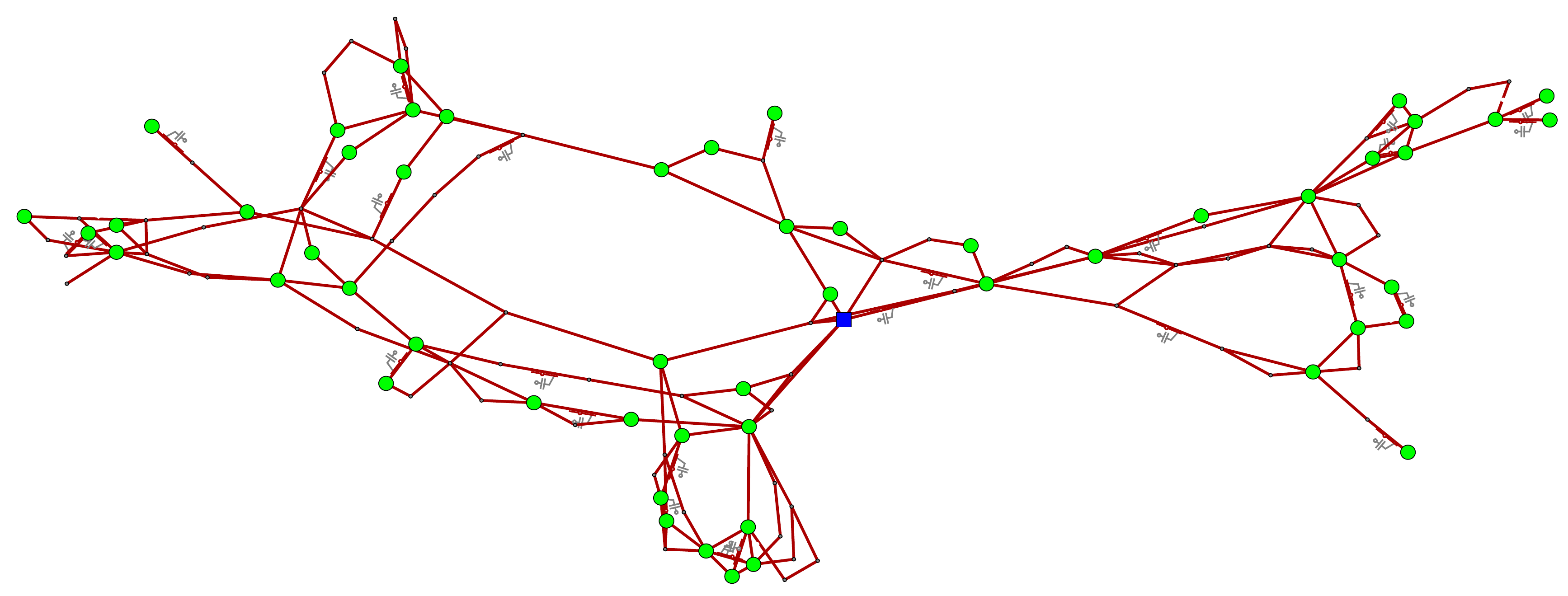}
\caption{Representation of the IEEE 118 Bus Test Case.  The network is composed of 118 buses, 186 lines, 54 generators.  The generators are marked by green circles and the reference generator is marked by the blue square.  The location of the 26 capacitor are shown.}
\label{fig-case118_Graph}
\end{figure*}

Let $\delta(t)$ be the 54 generator relative rotor phases in radians.  The state is $x(t) = [\delta(t),\dot{\delta}(t)]^T$ and the control is the scheduling of the synchronized switching of the capacitor switches.  The disturbance response problem is to schedule the capacitor switching that best improves network performance.  The disturbance is modeled as a perturbation to the initial rotor angles from steady state, $\delta_{ss}$. Such a disturbance may be due to a fault in the system or a quick change to system loads.  The initial condition is $\delta_0 = \delta_{ss}+\delta_{dist}$ where the disturbance $\delta_{dist}$ is a vector of random numbers from a uniform distribution between $[-0.3,0.3]$ radians.  

For the disturbance used in the example, integration of the swing equations reveals that without control, the system diverges from synchronized operation within 60 seconds (see Fig.~\ref{fig-power_net_results}(d) for no control). We provide results for two different approaches to reject the disturbance. The first is to compute through the projection-based mode scheduling, Algorithm~\ref{alg-iter_update}, the optimal schedule for the first $T = 5$ seconds following the disturbance.  The second is to apply a sliding window (receding horizon) approach of duration $T = 5$ seconds with a new schedule computed and applied each $0.1$ seconds.

\new{
Mode scheduling for power networks was executed on a core i7-3770K processor. The algorithm was implemented in C++. All differential equations were solved using GNU Scientific Libraries' implementation of explicit embedded Runge-Kutta Prince-Dormand 8, 9 method.
}

\subsection{Optimal Schedule}
We apply optimal mode scheduling, Algorithm~\ref{alg-iter_update}, to schedule the capacitor switches in order to respond to the disturbance. Let mode 1 be the dynamics with all capacitor switches ``off'' while mode 2 be the dynamics with all capacitor switches ``on''.  The cost is given by $\ell(x(t),u(t)) = 1/2 (\delta(t)-\bar{\delta}(t))^T(\delta(t)-\bar{\delta}(t)) + 1/40 (\dot{\delta}(t)-2\pi f_s)^T(\dot{\delta}(t)-2\pi f_s)$ where $\bar{\delta}(t)$ is the mean rotor phase at time $t$ and $f_s$ is the generator frequency.  The backtracking parameters are set to $\alpha = 0.4$ and $\beta = 0.1$.  

The results of mode scheduling the initial 5 seconds following a disturbance for 100 iterations of the algorithm are shown in Fig.~\ref{fig-power_net_results}. We find that the rotor phases do not diverge with the computed schedule. The cost reduces from $J = 170.68$ to $J = 54.78$ (see Fig.~\ref{fig-power_net_results}b), and the optimality function increases from $\theta = -2213.71$ to $\theta = -20.32$ (see Fig.~\ref{fig-power_net_results}a).  The total number of modes in the $7^\textrm{th}$ iteration's schedule is $M^7 = 66$, while the final switching schedule has $M^{100} = 120$. The schedules at the $7^\textrm{th}$ and $100^\textrm{th}$ iteration are in Fig.~\ref{fig-power_net_results}c.  

For the initial iterations in which $(x^k,u^k)$ are far from an infima, both the optimality function (see Fig.~\ref{fig-power_net_results}a) and the cost (see Fig.~\ref{fig-power_net_results}b) reduce significantly, which is a phenomenon that occurs with first-order smooth numerical optimization methods like steepest descent.  Since most of the gained performance occurs in the first few iterations, it is reasonable to expect that a sliding window real-time approach is viable.  Such an approach computes only the first few control synthesis iterations for each window.  

\begin{figure*}
\centering
\def\svgwidth{0.95\textwidth}
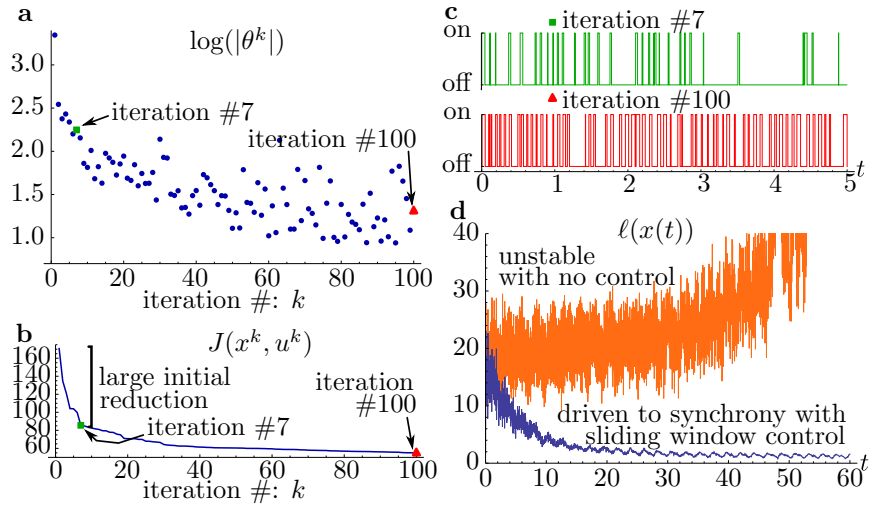
\caption{\textbf{a}, Convergence of optimality function toward zero as a function of iteration.  \textbf{b}, Cost $J$ as a function of iteration.  A large reduction occurs in the first 7 iterations. \textbf{c}, Comparison of the control signal for iterations 7 and 100. \textbf{d}, Comparison of running cost for no control (orange) and sliding window control (blue).}
\label{fig-power_net_results}
\end{figure*}

\subsection{Sliding Window Scheduling}

The second strategy is a switched system model predictive receding horizon control where the projection-based mode scheduling algorithm computes the schedule for each window of a sliding window. The receding horizon approach computes a schedule for a time window of duration $T = 5$ seconds but applies it for only $dt = 0.1$ seconds before incrementing the window $dt$ seconds and repeating for the new $5$ second time window.  The current window's initial state inherits the previous window's state at time $t_{i-1}+dt$. At each time $t_i = 0,0.1,0.2,\ldots$, we compute a single projection-based mode scheduling iteration for the time interval $t\in[t_i,t_i+T]$. The goal is for a real-time active control rejection of the disturbance.

Fig.~\ref{fig-power_net_results}d compares the running cost $\ell(\cdot)$ for the sliding window result against the no control result.  Without control, the system destabilizes, while sliding window single-bit control drives the system toward synchrony. Each window's schedule took on average 1.94 seconds to compute on a core i7-3770K processor. While the current implementation is 20 times slower than real-time, it indicates that an improved implementation on a more advanced computing machine could execute the receding horizon approach real-time even for the high-dimensional IEEE 118 Bus Test Case.

\section{Conclusion}
Optimal mode scheduling is an infinite dimensional, non-smooth problem.  The presented variational approach does not discretize in time or space a priori and as such may be applied to problems with a large number of states like the 108 state IEEE 118 Bus Test Case power network.  The presented algorithm, Algorithm~\ref{alg-iter_update}, parallels derivative-based iterative optimization in that it calculates a descent direction, takes a step of size calculated by an Armijo-like line search and updates.  As proven, if the descent direction is the negative mode insertion gradient and the step size is calculated from the backtracking algorithm, Algorithm~\ref{alg-backtracking}, then there are guarantees on convergence.  Furthermore, since the algorithm parallels standard derivative-based algorithms and since all of the calculations are provided in this paper, Algorithm~\ref{alg-iter_update} is easily implementable.  The mode scheduling algorithm is applied to the problem of power network regulation and a mode scheduling sliding window approach for near real-time control is demonstrated.

\new{
As discussed in the Introduction, there are two main themes to optimal mode scheduling: averaging methods \cite{bengea_decarlo, vasudevan_etal}, and insertion methods \cite{axelsson_wardi_et_al,egerstedt_wardi_axelsson,gonzalez_vasudevan_etal,wardi_egerstedt_hale,wardi_egerstedt_acc12}, where this paper is a member of the latter. Averaging methods can be more efficient at scheduling problems for which the solution spends more time chattering, while insertion methods can be better scheduling problems with isolated transitions. It would be interesting to synthesize an ``ideal'' mode scheduler that uses both. A possible approach is a superviser that iteratively divides the time horizon into intervals for which it is expected that the solution dominantly chatters and intervals where the solution dominantly does not and commands the preferred method to its preferred interval. 
}

\section{Acknowledgements}
This material is based upon work supported by the National Science Foundation under award CMMI-1200321 as well as the  Department of Energy Office of Science Graduate Fellowship Program (DOE SCGF), made possible in part by the American Recovery and Reinvestment Act of 2009, administered by ORISE-ORAU under contract no. DE-AC05-06OR23100.

\section{Appendix}
\subsection{Summary of Notation}
\begin{tabular}{|p{0.3\textwidth}p{0.65\textwidth}|}
\hline
$\cdot^-, \cdot^+$                       &                limit from the left, limit from right. \\
$\cdot_{ab}$                                   &                difference of index $a$ with index $b$, e.g. $d_{ab} = d_a-d_b$ \\
$[0,T]$                                           &                time interval  \\
$1:\setR\rightarrow\{0,1\}$          &               Heaviside step function  \\
$D$                                               &                slot derivative of single argument function---e.g. $Dg(y) = \frac{\partial}{\partial y}g(y)$ \\
$D_i$                                             &                slot derivative of $i^\textrm{th}$ argument---e.g. $D_2g(y,z) = \frac{\partial}{\partial z}g(y,z)$ \\
$d:[0,T]\rightarrow\setR^n$         &                 mode insertion gradient, Eq.~(\ref{eq-insertion_gradient})  \\
$F:\setR^n\times\setR^N\rightarrow \setR^n$ & switched system vector field Eq.~(\ref{eq-state})\\
$f_i:\setR^n\rightarrow \setR^n$                                &                 $i^\textrm{th}$ mode \\
$\gamma\in\setR^+$                   &                step size  \\
$J:\mathcal{X}\times\mathcal{U}\rightarrow\setR$ & cost function  \\
$k$                                                  &              iterate of sequence \\
$n$                                                  &              number of states \\
$N$                                                  &              number of distinct modes \\
$M$                                                  &              number of modes in $\Sigma$ \\
$\Omega\subset\mathcal{R}\subset \mathcal{U}$      &               set of non-chattering switching controls \\ 
$\mathcal{P}:\mathcal{X}\times\mathcal{R}\rightarrow\mathcal{S}$ & max-projection, Eq.~(\ref{eq-maxP})  \\
$\mathcal{Q}:\mathcal{R}\rightarrow\Omega$ & max-mapping, Eq.~(\ref{eq-maxProjStep})  \\
$\mathcal{R}\subset\mathcal{U}$ &               $\mathcal{R} = \mathcal{Q}^{-1}(\mathcal{S})$, see Assumption~\ref{ass-v_crit} \\
$\rho: [0,T]\rightarrow\setR^n$     &               adjoint, solution to Eq.~(\ref{eq-rhodot})  \\
$\mathcal{S}\subset \mathcal{X}\times\mathcal{U}$ &               set of feasible switched system trajectories, Def.~\ref{def-S}  \\
$\Sigma$                                       &               mode sequence \\
$(\Sigma,\mathcal{T})$                &                  mode schedule, Def.~\ref{def-sched} \\ 
$\sigma_i\in\{1,\ldots,N\}$          &               $i^\textrm{th}$ mode in $\Sigma$ \\
$T_i\in\setR$                                &                 $i^\textrm{th}$ switching time  \\
$\mathcal{T}$                              &                 vector of switching times \\
$\theta$                                         &              optimality function, Eq.~(\ref{eq-thetacalc})  \\
$\mathcal{U}$                                &              set of $N$ Lebesgue integrable functions  \\
$u: [0,T]\rightarrow\setR^N$         &              switching control, Def.~\ref{def-omega}  \\
$u_a: [0,T]\rightarrow\setR$           &              $a^\textrm{th}$ index of $u$ \\
$\mathcal{X}$                                &              set of $n$ Lebesgue integrable functions  \\
$x: [0,T]\rightarrow\setR^n$         &             state, solution to Eq.~(\ref{eq-state})  \\
\hline
\end{tabular}

\subsection{Proof of Lemma \ref{lem-lipschitz}: Lipschitz condition for $\ddot{d}_{ab}(t)$ \label{app-lipschitz}}
\begin{proof}
First, $d_{ab} = \rho(t)^T(f_a(x(t)) - f_b(x(t)))$.  Consider each $t\in(\tau_1,\tau_2)$, $x(t)\in\setR^n$ and $j\in\{1,\ldots,N\}$.  
Since $\|D^2f_j(x(t))\|<K_2$, there is a $K_0>0$ and $K_1>0$ such that $\|f_j(x(t))\|\leq K_0$ and $\|Df_j(x(t))\|\leq K_1$.  Therefore, 
$\|\dot{x}(t)\| = \|F(x(t),u(t))\|\leq K_0$ and 
\[
\begin{array}{l}
\|\ddot{x}(t)\| = \|D_1F(x(t),u(t))F(x(t),u(t))\| \\\hspace{20pt}\leq\|D_1F(x(t),u(t))\|\|F(x(t),u(t))\|\leq K_0K_1
\end{array}
\]
(recall that $\dot{u}(t) = 0$ for all $t\in(\tau_1,\tau_2)$).  

From Assumption \ref{ass-C2bnd}, $F(x(t),u(t))$ is Lipschitz in state since each $D^2f_j(x(t))$ is bounded. Additionally, $F(x,u)$ is piecewise continuous in $t$ since $u\in\Omega$ is non-chattering. Therefore, through Theorem 3.2 in \cite{khalil}, the solution to the state equations, Eq.~(\ref{eq-state}), exists over $[0,T]$ and is unique.  In order to signify that the solution $x$ depends on the fixed $u\in\Omega$, we write $x(t;u)$.

Define $g(\rho(t)) := -D_1F(x(t;u),u(t))^T\rho(t) - D\ell(x(t;u))$, which exists for any $t\in(\tau_1,\tau_2)$.  Since $\|D_1F(x(t;u),u(t))\|\leq K_1$, we see that $g(\cdot)$ is Lipschitz with respect to $\rho$:
\begin{equation}
\begin{array}{l}
\|g(\rho_2(t)) - g(\rho_1(t))\| \leq \|D_1F(x(t;u),u(t))^T\|\|\rho_2(t) - \rho_1(t)\| \\\hspace{10pt} \leq K_1\|\rho_2(t) - \rho_1(t)\|.
\end{array}
\label{eq-g_lip}
\end{equation}
Furthermore, $g(\rho)$ is piecewise continuous in $t$ because $D_1F$ and $D\ell$ are both piecewise continuous since $u\in\Omega$ is non-chattering. Since $g(\rho)$ is Lipschitz in $\rho$ and piecewise continuous in $t$, the solution to Eq.~(\ref{eq-rhodot})---i.e. the solution to $\dot{\rho} = g(\rho)$---exists  over $[0,T]$ and is unique through Theorem 3.2 in \cite{khalil}.  Label this solution as $\rho(t;u)$ due to its dependence on $u\in\Omega$.

Due to the existence of $\rho(t;u)$, there is a $K'_0>0$ such that $\rho(t;u)\leq K'_0$.  Additionally, since $D\ell(x(t;u))$ is bounded, through the triangle inequality, there is a $K'_1>0$ such that, $\|\dot{\rho}(t;u)\| = \|g(\rho(t;u))\|\leq K'_1$.  By this bound, it follows that for each $t_1,t_2\in(\tau_1,\tau_2)$, $\|\rho(t_2;u)-\rho(t_1;u)\|\leq K'_1|t_2-t_1|$.  From this condition and Eq.~(\ref{eq-g_lip}), there is $L_1$ such that $\|\dot{\rho}(t_2;u)-\dot{\rho}(t_1;u)\|\leq L_1|t_2-t_1|$.  Note,
\[
\begin{array}{l}
\ddot{\rho}(t;u) = -F(x(t;u),u(t))^TD_1^2F(x(t;u),u(t))\rho(t;u) \\\hspace{10pt}- D_1F(x(t;u),u(t))^T\dot{\rho}(t;u) - D^2\ell(x(t;u))F(x(t;u),u(t)),
\end{array}
\]
By the bounds on $F(\cdot,\cdot)$, $D_1F(\cdot,\cdot)$, $D_1^2F(\cdot,\cdot)$, and $D^2\ell$, and that $\rho(t;u)$ and $\dot{\rho}(t;u)$ are Lipschitz, for any times $t_1,t_2\in(\tau_1,\tau_2)$, through the triangle inequality, there is $L_2>0$ such that $\|\ddot{\rho}(t_2;u)-\ddot{\rho}(t_1;u)\|\leq L_2|t_2-t_1|$. Finally, 
\begin{equation}
\begin{array}{l}
\ddot{d}_{ab}(t) = \ddot{\rho}(t;u)^T(f_a(x(t;u)) - f_b(x(t;u))) \\\hspace{10pt}+2\dot{\rho}(t;u)^T(Df_a(x(t;u))-Df_b(x(t;u)))\dot{x}(t;u)\\\hspace{10pt}+ \rho(t;u)^T(D^2f_a(x(t;u))-D^2f_b(x(t;u)))\circ(\dot{x}(t;u), \dot{x}(t;u))\\\hspace{10pt}+ \dot{\rho}(t;u)^T(Df_a(x(t;u))-Df_b(x(t;u))) \ddot{x}(t;u)
\end{array}
\label{eq-dab_doubledot}
\end{equation}
where $\circ$ is the bilinear operator of $D^2f_a - D^2f_b$. By the bounds on $f_a$, $f_b$, and their first and second derivatives, the bounds on $\dot{x}$ and $\ddot{x}$, as well as the Lipschitz condition with respect to $t$ of $\rho(t;u)$, $\dot{\rho}(t;u)$ and $\ddot{\rho}(t;u)$ it is the case that $\ddot{d}_{ab}$ is Lipschitz with some constant $L>0$ in the interval $(\tau_1,\tau_2)$.  
\end{proof}

\subsection{Proof of Lemma \ref{lem-gamma0}: $\gamma_0$ \label{app-gamma0}}
\begin{proof}
Choose any time $t\in(T_{i-1},T_i)$, $i\in\{1,\ldots,M\}$.  The active mode at time $t$ is $\sigma_i$. In the switching control representation, the $\sigma_i^\textrm{th}$ index of $u$ is 1, $u_{\sigma_i}(t) = 1$, while each other mode $a\in\{1,\ldots,N\}$, $a\neq \sigma(t)$, is $u_a(t) = 0$. Reusing the double subscript notation, define $u_{a\sigma_i}(t):=u_a(t) - u_{\sigma_i}(t)$, which is $u_{a\sigma_i}(t) = -1$. 

Furthermore, note $d_{\sigma_i}(t)$ is the sensitivity of the cost from inserting the active mode. This insertion is equivalent to switching the mode schedule from $\sigma_i$ to $\sigma_i$ for an infinitesimal period of time, which has no effect to the control or cost. Therefore, the cost is not sensitive to inserting the active mode and as such, $d_{\sigma_i}(t) = 0$, which is also realized from Eq.~(\ref{eq-insertion_gradient}).  Therefore, $d_{a\sigma_i}(t) := d_a(t)-d_{\sigma_i}(t) = d_a(t)$.

Recall the definition of the max-mapping $\mathcal{Q}$ where $\mathcal{Q}_a(u(t)-\gamma d(t)) := \prod^N_{b\neq a}1(u_a(t)-\gamma d_a(t) - (u_b(t)-\gamma d_b(t))) = \prod^N_{b\neq a}1(u_{ab} - \gamma d_{ab})$ for each $a\in\{1,\ldots,N\}$.  At any time $t\in(T_{i-1},T_i)$, $i\in\{1,\ldots,M\}$, for $a$ to be the active mode of $\mathcal{Q}(u(t)-\gamma d(t))$, but not the active mode of $u$---i.e. $a\neq \sigma_i$---it must be the case that $u_{ab} - \gamma d_{ab}> 0$ for each $b\neq a$ including $b = \sigma_i$.  It follows that it is necessary for $u_{a\sigma_i}(t) - \gamma d_{a\sigma_i}(t)>0$ for mode $a$ to be active at time $t$. When $\gamma = 0$, $u_{a\sigma_i} = -1$, and therefore $\mathcal{Q}_a(u(t)) = 0$ for all $a\neq\sigma_i$ and therefore $\mathcal{Q}(u) = u$. 

Assuming $a$ is active for some $\gamma > 0$, then it must be the case that $u_{a\sigma(t)}(t)-\gamma d_{a\sigma(t)}(t) > 0$. Simplifying using $u_{a\sigma_i}(t) = -1$ and $d_{a\sigma_i}(t) = d_a$, it must be the case that 
\begin{equation}
\gamma d_a(t) < -1.
\label{eq-gamma_ineq}
\end{equation}
If $d_a(t)$ is negative, then there is a $\gamma>0$ such that the inequality is true. However, if for each $a\in\{1,\ldots,N\}$, $d_a(t)$ is aways non-negative, then the inequality can never be true and therefore, $\mathcal{Q}(u-\gamma d) = u$ for all $\gamma>0$. 

Assuming there exists an $a\in\{1,\ldots,N\}$ and $t\in[0,T]$ such that $d_a(t)$ is negative consider $\gamma_0$, Eq.~(\ref{eq-gamma0k}). Take a mode $a$ and time $t$ such that $d_a(t) = \theta$, from Eq.~(\ref{eq-thetacalc}). We see $\theta$ is finite since $d_a(t)$ is finite through Assumption~\ref{ass-C2bnd}, and therefore, $\gamma_0 = -\frac{1}{d_a(t)}$ is finite. When $\gamma = \gamma_0^+$, the inequality in Eq.~(\ref{eq-gamma_ineq}) is true and therefore, $u_{a\sigma(t)}(t)-\gamma_0^+ d_{a\sigma(t)}(t) > 0$ and $\mathcal{Q}(u-\gamma_0^+ d)\neq u$. 

Finally, for any $\gamma\in[0,\gamma_0)$, in order for $\mathcal{Q}(u-\gamma d)\neq u$, there must be a mode $b$ and time $\tau$ such that $d_b(\tau) < -1/\gamma$. however, $-1/\gamma < -1/\gamma_0 = \theta$ and by the definition of $\theta$, there is no $d_b(\tau)<\theta$. Due to the contradiction, there is no $\gamma\in[0,\gamma_0)$ such that $\mathcal{Q}(u-\gamma d)\neq u$.
\end{proof}

\subsection{Proof of Lemma \ref{lem-continuity_of_T}: Continuity of switching times \label{app-continuity_of_T}}
\begin{proof}
Assume there is no $\overline{\delta\gamma}$ such that $T_i(\gamma)$ is continuous in $(\gamma_0,\gamma_0+\overline{\delta\gamma})$. Then, for every $\overline{\delta\gamma}\in(0,\delta\gamma)$, there is a $\gamma\in(\gamma_0,\gamma_0+\overline{\delta\gamma})$ such that $DT_i(\gamma)$ does not exist.  According to Lemma~\ref{lem-DTgamma}, $DT_i(\gamma)$ exists when $T_i(\gamma)$ is not a critical point of $\mu_{\sigma_i\sigma_{i+1}}:=u_{\sigma_i\sigma_{i+1}} - \gamma d_{\sigma_i\sigma_{i+1}}$.  Therefore, for any $\delta t\in\setR^+$, there must be a $t\in (T_i(\gamma_0^+),T_i(\gamma_0^+)+\delta t)$ such that $t$ is a critical point of $\mu_{\sigma_i\sigma_{i+1}}$---i.e. there is a $t$ such that $\mu_{\sigma_i\sigma_{i+1}}$ is discontinuous at $t$ or $\dot{d}_{\sigma_i\sigma_{i+1}}(t) = 0$.  The following shows that there are $\delta t$ for which no $\mu_{\sigma_i\sigma_{i+1}}(t)$, $t\in  (T_i(\gamma_0^+),T_i(\gamma_0^+)+\delta t)$ is a critical point and so the assumption that $\overline{\delta\gamma}$ does not exist is false.

By the assumption that $u$ is piecewise constant, there is a $\overline{\delta t}\in\setR\setminus 0$ such that $u(t)$ is constant for $t\in(T_i(\gamma_0^+),T_i(\gamma_0^+)+\overline{\delta t})$, whether $T_i(\gamma_0^+)$ is a point of discontinuity of $u$ or not. Since $u$ is constant in the time interval, by Lemma~\ref{lem-lipschitz}, $\ddot{d}_{\sigma_i\sigma_{i+1}}(t)$ is also Lipschitz in the interval.  

If $m(T_i(\gamma_0^{+})) = 1$, then $\dot{d}_{\sigma_i\sigma_{i+1}}(T_i(\gamma_0))\neq 0$ and by Lipschitz of $\ddot{d}_{\sigma_i\sigma_{i+1}}$, there is a $\delta t\in (0,\overline{\delta t})$ such that for each $t\in(T_i(\gamma_0^+),T_i(\gamma_0^+)+\delta t)$,  $\dot{d}_{\sigma_i\sigma_{i+1}}(t)\neq 0$.  Additionally, by the continuity of $\dot{d}_{\sigma_i\sigma_{i+1}}$ in the time interval, $\dot{d}_{\sigma_i\sigma_{i+1}}(t)$ does not change signs in $(T_i(\gamma_0^+),T_i(\gamma_0^+)+\delta t)$.  Thus, there is a $\delta t$ such that for every $t\in (T_i(\gamma_0^+),T_i(\gamma_0^+)+\delta t)$, $t$ is not a critical point of $\mu_{\sigma_i\sigma_{i+1}}(t)$ and therefore, by contradiction, there is a $\overline{\delta\gamma}\in(0,\delta\gamma]$ such that $T_i:(\gamma_0,\gamma_0+\overline{\delta \gamma})\rightarrow (T_i(\gamma_0^+),T_i(\gamma_0^+)+\delta t)$ is continuous. Furthermore, by referring to Eq.~(\ref{eq-Dtau_cont}), since the sign of $\dot{d}_{\sigma_i\sigma_{i+1}}(t)$, $t\in (T_i(\gamma_0^+),T_i(\gamma_0^+)+\delta t)$, is constant, the sign of $DT_i(\gamma)$, $\gamma \in (\gamma_0,\gamma_0+\overline{\delta \gamma})$, is constant and so $T_i: (\gamma_0,\gamma_0+\overline{\delta \gamma})$ is strictly monotonic.

If $m^k(T_i(\gamma_0^{+})) = 2$, then $\dot{d}_{\sigma_i\sigma_{i+1}}(T_i(\gamma_0^+)) = 0$ but $\ddot{d}_{\sigma_i\sigma_{i+1}}(T_i(\gamma_0^+)) \neq 0$. Since $\ddot{d}_{\sigma_i\sigma_{i+1}}(t)$ is Lipschitz for $t\in(T_i(\gamma_0^+),T_i(\gamma_0^+)+\overline{\delta t})$, there is a $\delta t\in (0,\overline{\delta t})$ such that for each $t\in(T_i(\gamma_0^+),T_i(\gamma_0^+)+\delta t)$, $d_{\sigma_i\sigma_{i+1}}(t)$ and $\dot{d}_{\sigma_i\sigma_{i+1}}(t)$ are strictly monotonic. Consequently, $\dot{d}_{\sigma_i\sigma_{i+1}}(t)\neq 0$ for $t\in(T_i(\gamma_0^+),T_i(\gamma_0^+)+\delta t)$ and thus by contradiction, there is a $\overline{\delta\gamma}\in(0,\delta\gamma]$ such that $T_i:(\gamma_0,\gamma_0+\overline{\delta \gamma})\rightarrow (T_i(\gamma_0^+),T_i(\gamma_0^+)+\delta t)$ is continuous. Furthermore, since $\dot{d}_{\sigma_i\sigma_{i+1}}(T_i(\gamma_0^+)) = 0$ and $\dot{d}_{\sigma_i\sigma_{i+1}}(t)$ in $t\in (T_i(\gamma_0^+),T_i(\gamma_0^+)+\delta t)$ is strictly monotonic, the sign of $\dot{d}_{\sigma_i\sigma_{i+1}}(t)$, $t\in (T_i(\gamma_0^+),T_i(\gamma_0^+)+\delta t)$ is constant. By referring to Eq.~(\ref{eq-Dtau_cont}), the sign of $DT_i(\gamma)$, $\gamma \in (\gamma_0,\gamma_0+\overline{\delta \gamma})$, is constant and so $T_i: (\gamma_0,\gamma_0+\overline{\delta \gamma})$ is strictly monotonic.
\end{proof}

\subsection{Proof of Lemma \ref{lem-d_properties}: Dependence of $d(\cdot)$ on $T_i(\gamma)$
\label{app-d_properties}}
\begin{proof}
Set $\mu = u-\gamma d$. For $T_i(\gamma)$ to be a switching time between modes $\sigma_i$ and $\sigma_{i+1}$, the index of $\mu$ with greatest value must switch from $\sigma_i$ to $\sigma_{i+1}$ at $T_i(\gamma)$. By the definition of the max-mapping $\mathcal{Q}(\cdot)$, Eq.~(\ref{eq-maxProjStep}), $\mu_{\sigma_i\sigma_{i+1}}(T_i(\gamma)^-)$ must be negative while $\mu_{\sigma_i\sigma_{i+1}}(T_i(\gamma)^+)$ must be positive and If $\mu_{\sigma_i\sigma_{i+1}}$ is continuous at $T_i(\gamma)$, then $\mu_{\sigma_i\sigma_{i+1}}(T_i(\gamma)) = 0$. We show $\mu_{\sigma_i\sigma_{i+1}}$ is continuous at $T_i(\gamma)$ through Lemma~\ref{lem-continuity_of_T}.  According to Lemma~\ref{lem-continuity_of_T}, there is a $0<\delta \gamma'\leq\delta \gamma$ such that for $\gamma\in(\gamma_0,\gamma_0+\delta\gamma')$, $T_i(\gamma)$ is continuous and strictly monotonic.  Since $u\in\Omega$ is non-chattering and $d_{\sigma_i\sigma_{i+1}}$ has a finite number of critical points as assumed in Assumption~\ref{ass-v_crit}, there is a $0<\overline{\delta \gamma}_1\leq\delta \gamma_1$ such that for each $\gamma\in(\gamma_0,\gamma_0+\overline{\delta\gamma}_1)$, $u_{\sigma_i\sigma_{i+1}}(T_i(\gamma))$ is constant in $T_i(\gamma)$ and therefore $\ddot{d}(T_i(\gamma))$ is Lipschitz continuous in $T_i(\gamma)$ through Lemma~\ref{lem-lipschitz}. Therefore, $\mu_{\sigma_i\sigma_{i+1}}(T_i(\gamma))$ is continuous in $T_i(\gamma)$ and we conclude
\begin{equation}
u_{\sigma_i\sigma_{i+1}}(T_i(\gamma))-\gamma d_{\sigma_i\sigma_{i+1}}(T_i(\gamma)) = 0.
\label{eq-u_minus_gamma_d}
\end{equation}
Eq.~(\ref{eq-u_minus_gamma_d}) can be simplified depending on whether $\sigma_i$ or $\sigma_{i+1}$ is the active mode of $u$ at $T_i(\gamma)$. If $\sigma_i$ is the active mode, then $u_{\sigma_i}(T_i(\gamma)) = 1$ and $u_{\sigma_{i+1}}(T_i(\gamma)) = 0$ and therefore, $u_{\sigma_i\sigma_{i+1}}(T_i(\gamma)) = 1$. Additionally, referring to Eq.~(\ref{eq-insertion_gradient}), the mode insertion gradient of the active mode has value 0 and so here, $d_{\sigma_i}(T_i(\gamma)) = 0$. Therefore, $d_{\sigma_i\sigma_{i+1}}(T_i(\gamma)) = -d_{\sigma_{i+1}}(T_i(\gamma))$. Similarly, if $\sigma_{i+1}$ is the active mode of $u$ at $T_i(\gamma)$, then $u_{\sigma_i\sigma_{i+1}}(T_i(\gamma)) = -1$ and $d_{\sigma_i\sigma_{i+1}}(T_i(\gamma)) = d_{\sigma_{i}}(T_i(\gamma))$. Plugging into Eq.~(\ref{eq-u_minus_gamma_d}), either
\begin{equation}
-1-\gamma d_{\sigma_{i+1}}(T_i(\gamma)) = 0\textrm{ or } -1-\gamma d_{\sigma_{i}}(T_i(\gamma)) = 0
\label{eq-1_minus_gamma_d_two_cases}
\end{equation}
when $\sigma_i$ or $\sigma_{i+1}$ is the active mode of $u$ at $T_i(\gamma)$ respectively.  Finally, if $T_i(\gamma)$ is increasing in value---i.e. $\omega = 0$---then $\sigma_{i+1}$ is the active mode of $u$ at $T_i(\gamma)$, while if $T_i(\gamma)$ is decreasing in value---i.e. $\omega = 1$---then $\sigma_{i}$ is the active mode of $u$ at $T_i(\gamma)$. Eq.~(\ref{eq-1_minus_gamma_d_two_cases}) reduces to Eq.~(\ref{eq-Tigamma_form}).

To prove point 2 of the Lemma, take the derivative of Eq.~(\ref{eq-Tigamma_form}) with respect to $\gamma$:
\begin{equation}
-d_{\sigma_{i+\omega}}(T_i(\gamma)) -\gamma \dot{d}_{\sigma_{i+\omega}}(T_i(\gamma))DT_i(\gamma) = 0
\label{eq-1_minus_gamma_d_two_cases_dot}
\end{equation}
which is possible since $\ddot{d}_{\sigma_{i+\omega}}(T_i(\gamma))$ is Lipschitz continuous in $T_i(\gamma)$ and $DT_i(\gamma)$ exists through Lemma~\ref{lem-DTgamma}. By the continuity of $d_{\sigma_{i+\omega}}(T_i(\gamma))$ and $\dot{d}_{\sigma_{i+\omega}}(T_i(\gamma))$ in $T_i(\gamma)$, there is a $0<\overline{\delta\gamma}_2\leq\overline{\delta\gamma}_1$ such that for $\gamma \in (\gamma_0,\gamma_0+\overline{\delta\gamma}_2)$, $d_{\sigma_{i+\omega}}(T_i(\gamma))$ is negative and $(-1)^\omega DT_i(\gamma)$ is positive---i.e. $DT_i(\gamma)$ is positive when $\omega = 0$ and negative when $\omega = 1$. Therefore, for the equality in Eq.~(\ref{eq-1_minus_gamma_d_two_cases_dot}) to be true, $(-1)^\omega \dot{d}_{\sigma_{i+\omega}}(T_i(\gamma)) >0$.

The Lemma's point 3 follows from the Mean Value Theorem and the Lemma's point 2. Recall when $m(T_i(\gamma_0^+)) = 0$,  $\dot{d}_{\sigma_{i+\omega}}(T_i(\gamma_0^+)) = 0$, but $\ddot{d}_{\sigma_{i+\omega}}(T_i(\gamma_0^+)) \neq 0$. Since $\ddot{d}_{\sigma_{i+\omega}}(T_i(\gamma))$ is Lipschitz in $T_i(\gamma)$, there is a $0<\overline{\delta\gamma}_3\leq\overline{\delta\gamma}_2$ such that for each $\gamma \in (\gamma_0,\gamma_0+\overline{\delta\gamma}_3)$, $\ddot{d}_{\sigma_{i+\omega}}(T_i(\gamma))\neq 0$. Consider $\dot{d}_{\sigma_{i+\omega}}(T_i(\gamma_0+\overline{\delta\gamma}_3))$.
Through the Mean Value Theorem, there is a $\gamma'$ such that 
\[
\begin{array}{l}
\ddot{d}_{\sigma_{i+\omega}}(T_i(\gamma')) = \frac{\dot{d}_{\sigma_{i+\omega}}(T_i(\gamma_0+\overline{\delta\gamma}_3)) - \dot{d}_{\sigma_{i+\omega}}(T_i(\gamma_0^+))}{T_i(\gamma_0 + \overline{\delta\gamma}_3) - T_i(\gamma_0^+)}\\
\hspace{60pt}
 = \frac{\dot{d}_{\sigma_{i+\omega}}(T_i(\gamma_0+\overline{\delta\gamma}_3)) }{T_i(\gamma_0 + \overline{\delta\gamma}_3) - T_i(\gamma_0^+)}
\end{array}
\]
Since, according to the Lemma's point 2, $(-1)^\omega\dot{d}_{\sigma_{i+\omega}}(T_i(\gamma_0+\overline{\delta\gamma}_3))>0$, and that $(-1)^\omega(T_i(\gamma_0 + \overline{\delta\gamma}_3) - T_i(\gamma_0^+))>0$, $\ddot{d}_{\sigma_{i+\omega}}(T_i(\gamma'))>0$. Since the sign of $\ddot{d}_{\sigma_{i+\omega}}(T_i(\gamma))$ is constant for all $\gamma \in (\gamma_0,\gamma_0+\overline{\delta\gamma}_3)$, including at $T_i(\gamma')$, point 3 is true. Finally, set $\overline{\delta\gamma} = \overline{\delta\gamma}_3$.
\end{proof}

\subsection{Proof of Lemma \ref{lem-STapprox}: Local approximation of switching times \label{app-STapprox}}
\begin{proof}
For $m(T_i(\gamma_0^+)) = 1$, Eq.~(\ref{eq-Ti_approx_mk1}) follows from Taylor expanding $T_i(\gamma)$ around $\gamma_0^+$:
\begin{equation}
T_i(\gamma) = T_i(\gamma_0^+) + DT_i(\gamma_0^+)(\gamma-\gamma_0) + o(\gamma-\gamma_0).
\label{eq-Tigamma_init}
\end{equation}
Since $T_i(\gamma_0^+)$ is not a critical time of $u_{\sigma_i\sigma_{i+1}}-\gamma_0d_{\sigma_i\sigma_{i+1}}$, $DT_i(\gamma_0^+)$ is given in Eq.~(\ref{eq-Dtau_cont}). Furthermore, by the continuity and strict monotonicity of $T_i(\gamma)$ from Lemma~\ref{lem-continuity_of_T} and that $u\in\Omega$ is non-chattering, there is a $\overline{\delta\gamma}\in(0,\delta\gamma]$ such that $u_{\sigma_i\sigma_{i+1}}(T_i(\gamma))$ is constant and equal to $u_{\sigma_i\sigma_{i+!}}(T_i(\gamma_0^+))$. Thus,
\[
DT_i(\gamma_0^+) = \frac{u_{\sigma_i\sigma_{i+1}}(T_i(\gamma_0^+))}{\gamma_0^2\dot{d}_{\sigma_i\sigma_{i+1}}(T_i(\gamma_0^+))}.
\]
When $\omega = 0$, the active mode of $u$ at time $T_i(\gamma_0^+)$ is $\sigma_{i+1}$, so $u_{\sigma_i}(T_i(\gamma_0^+)) = 0$, $u_{\sigma_{i+1}}(T_i(\gamma_0^+)) = 1$, and $d_{\sigma_{i+1}}(T_i(\gamma_0^+)) = 0$. Therefore, $u_{\sigma_i\sigma_{i+1}}(T_i(\gamma_0^+)) = -1$ and $\dot{d}_{\sigma_i\sigma_{i+1}}(T_i(\gamma_0^+)) = \dot{d}_{\sigma_i}(T_i(\gamma_0^+))$. Similarly, when $\omega = 1$, the active mode of $u$ at time $T_i(\gamma_0^+)$ is $\sigma_{i}$, so $u_{\sigma_i}(T_i(\gamma_0^+)) = 1$, $u_{\sigma_{i+1}}(T_i(\gamma_0^+)) = 0$, and $d_{\sigma_{i}}(T_i(\gamma_0^+)) = 0$. Therefore, $u_{\sigma_i\sigma_{i+1}}(T_i(\gamma_0^+)) = 1$ and $\dot{d}_{\sigma_i\sigma_{i+1}}(T_i(\gamma_0^+)) = -\dot{d}_{\sigma_i}(T_i(\gamma_0^+))$. Plugging into $DT_i(\gamma_0^+)$,
\[
DT_i(\gamma_0^+) = -\frac{1}{\gamma_0^2\dot{d}_{\sigma_{i+\omega}}(T_i(\gamma_0^+))}.
\]
Plugging $DT_i(\gamma_0^+)$ into Eq.~(\ref{eq-Tigamma_init}) and setting $\theta = -1/\gamma_0$, Eq.~(\ref{eq-thetacalc}), results in Eq.~(\ref{eq-Ti_approx_mk1}).

For $m(T_i(\gamma_0^+)) = 2$, $\dot{d}_{\sigma_i\sigma_{i+1}}(T_i(\gamma_0^+)) = 0$ and so it is not possible to Taylor expand $T_i(\gamma)$ around $\gamma_0^+$ because $DT_i(\gamma)$, Eq.~(\ref{eq-Dtau_cont}), goes unbounded as $T_i(\gamma)$ approaches $T_i(\gamma_0^+)$. Instead, we start from the switching time equation, Eq.~(\ref{eq-Tigamma_form}) of Lemma~\ref{lem-d_properties}
\[
-1-\gamma d_{\sigma_{i+\omega}}(T_i(\gamma)) = 0.
\]
Through Lemmas~\ref{lem-continuity_of_T}, there exists a $\delta\gamma_1>0$ such that this implicit equation on $T_i(\gamma)$ is true for $\gamma\in(\gamma_0,\gamma_0+\delta\gamma_1)$ and $T_i(\gamma)$ is continuous and strictly monotonic. There is a $\delta t\in\setR/0$ such that $T_i:(\gamma_0,\gamma_0+\delta\gamma_1)\rightarrow (T_i(\gamma_0^+),T_i(\gamma_0^+)+\delta t)$.

Taylor expand $d_{\sigma_{i+\omega}}(T_i(\gamma))$ around $T_i(\gamma_0^{+})$ for $T_i(\gamma) \in (T_i(\gamma_0^+),T_i(\gamma_0^+)+\delta t)$, recalling that $\dot{d}_{\sigma_{i+\omega}}(T_i(\gamma_0^{+})) = 0$:
\[
\begin{array}{l}
d_{\sigma_{i+\omega}}(T_i(\gamma)) = d_{\sigma_{i+\omega}}(T_i(\gamma_0^{+})) + \frac{1}{2}\ddot{d}_{\sigma_{i+\omega}}(T_i(\gamma_0^{+})) \tau(\gamma)^2 \\\hspace{20pt}+ o(\tau(\gamma)^2)
\end{array}
\]
where $\tau(\gamma) = T_i(\gamma) - T_i(\gamma_0^+)$. Plug the expanded $d_{\sigma_{i+\omega}}(T_i(\gamma))$ into Eq.~(\ref{eq-Tigamma_form}) and reorder
\[
\frac{1}{2}\ddot{d}_{\sigma_{i+\omega}}(T_i(\gamma_0^{+}))\tau(\gamma)^2 = \frac{\gamma-\gamma_0}{\gamma\gamma_0} +  o(\tau(\gamma)^2).
\]
Taylor expanding $\frac{\gamma-\gamma_0}{\gamma\gamma_0}$ around $\gamma_0$, 
\begin{equation}
\frac{1}{2}\ddot{d}_{\sigma_{i+\omega}}(T_i(\gamma_0^{+}))\tau(\gamma)^2 = \frac{\gamma-\gamma_0}{\gamma_0^2} +  o(\gamma-\gamma_0) + o(\tau(\gamma)^2).
\label{eq-approx_taylor_form}
\end{equation}
By the Taylor expansion of $d_{\sigma_{i+\omega}}(\cdot)$ around $T_i(\gamma_0^{+})$, $o(\tau(\gamma)^2)$ is of lesser order than $\frac{1}{2}\ddot{d}_{\sigma_{i+\omega}}(T_i(\gamma_0^{+}))\tau(\gamma)^2$.  In order for the equality of Eq. (\ref{eq-approx_taylor_form}) to be true, $o(\tau(\gamma)^2)$ must also be of lesser order than  $\gamma-\gamma_0$.  Therefore, $o(\tau(\gamma)^2) = o(\gamma-\gamma_0)$.  Recall that $\theta = -1/\gamma_0$ and that $\ddot{d}_{\sigma_{i+\omega}}(T_i(\gamma_0^{+}))\neq 0$ since $m(T_i(\gamma_0^+))=2$. Solve for $\tau(\gamma)^2$:
\begin{equation}
\tau(\gamma)^2 = \frac{2\theta^2}{\ddot{d}_{\sigma_{i+\omega}}(T_i(\gamma_0^{+}))}(\gamma-\gamma_0)+ o(\gamma-\gamma_0).
\label{eq-approx_taylor_form_2}
\end{equation}
Set $c = \frac{2\theta^2 }{\ddot{d}_{\sigma_i\sigma_{i+1}}(T_i(\gamma_0^{+}))}$.  There is $0<\delta \gamma_2\leq\delta \gamma_1$ such that for $\gamma\in(\gamma_0,\gamma_0+\delta\gamma_2)$, 
\[
|c(\gamma-\gamma_0)| > o(\gamma-\gamma_0).
\]
Since $\ddot{d}_{\sigma_{i+\omega}}(T_i(\gamma_0^{+}))>0$ due to point 3 of Lemma~\ref{lem-d_properties}, the right side of Eq.~(\ref{eq-approx_taylor_form_2}) has a single positive real valued square root and a single negative real valued square root for each $\gamma\in(\gamma_0,\gamma_0+\delta\gamma_2)$. The switching time $T_i(\gamma)$ corresponds to one of the roots.  All that remains is to show that 
\[
[c(\gamma-\gamma_0)+o(\gamma-\gamma_0)]^\frac{1}{2} = c^\frac{1}{2}(\gamma-\gamma_0)^\frac{1}{2} + o((\gamma-\gamma_0)^\frac{1}{2}).
\]
In other words, we need to show that $[c(\gamma-\gamma_0)+o((\gamma-\gamma_0)]^\frac{1}{2}) - c^\frac{1}{2}(\gamma-\gamma_0)^\frac{1}{2} \in o((\gamma-\gamma_0)^\frac{1}{2})$.  By the definition of $o$, for each $p>0$, there is $0<\overline{\delta\gamma}(p)\leq\delta\gamma_4$ such that for all $0<\delta\gamma<\overline{\delta\gamma}(p)$, 
\[
\begin{array}{l}
[c\delta\gamma+o(\delta\gamma)]^\frac{1}{2} - c^\frac{1}{2}\delta\gamma^\frac{1}{2}<[c\delta\gamma+p\delta\gamma]^\frac{1}{2}-c^\frac{1}{2}\delta\gamma^\frac{1}{2} = [(c+p)^\frac{1}{2} - c^\frac{1}{2}]\delta\gamma^\frac{1}{2}.
\end{array}
\]
Set $p_2 = (c+p)^\frac{1}{2} - c^\frac{1}{2}$ which is zero when $p = 0$.  As such, for all $p_2>0$ and all $\delta\gamma\in\overline{\delta\gamma}(p)$, it is the case that $[c\delta\gamma+o(\delta\gamma)]^\frac{1}{2} - c^\frac{1}{2}\delta\gamma^\frac{1}{2}<p_2\delta\gamma$ and thus the left hand side of the inequality is an element of $o((\gamma-\gamma_0^k)^\frac{1}{2})$.
\end{proof}

\subsection{Proof of Lemma \ref{lem-Japprox}: Local approximation of the Cost \label{app-Japprox}}
\begin{proof}
The first order approximation of $J(\gamma)$ with respect to $\tau(\gamma):=\mathcal{T}(\gamma) - \mathcal{T}(\gamma_0^+)$ is $\tilde{J}(\gamma)$, Eq.~(\ref{eq-Japprox})---i.e. 
\[
J(\gamma) = \tilde{J}(\gamma) + o(|\tau(\gamma)|).
\]
The approximation $\hat{J}(\overline{m};\gamma)$ is a further approximation from $\tilde{J}(\gamma)$, which includes the approximation of $\tau(\gamma)_i := T_i(\gamma)-T_i(\gamma_0)$ using Lemma~\ref{lem-STapprox}.  Consider $\overline{m} = 1$ first.  All switching times in $\mathcal{T}(\gamma_0^+)$ are type-0 or type-1. Set $H = (I_1)^c$ as the complement of $I_1$---i.e. the index set of type-0 switching times.  Each $T_h(\gamma_0^+)$ for $h\in H$ is type-0 and so $\tau(\gamma)_h = 0$.  As for the type-1 switching times, $\tau_i(\gamma)$ is approximated by Eq.~(\ref{eq-Ti_approx_mk1}) where $i\in I_1$. As such, the full vector $\tau(\gamma)$ approximately varies linearly with $\gamma-\gamma_0$ and so $o(|\tau(\gamma)|) = o(\gamma-\gamma_0)$. Plugging $\tau_i(\gamma)$ from Eq.~(\ref{eq-Ti_approx_mk1}) for each $i\in I_1$ into $\tilde{J}(\gamma)$ results in $J(\gamma) = \hat{J}(1;\gamma) + o(\gamma-\gamma_0)$. Therefore, $R(\gamma) = o(\gamma-\gamma_0)$ and $|\hat{J}(1;\gamma)-J(0)| \geq  |R(\gamma)|$.

Now for the case where $\overline{m} = 2$.  First, the approximations of $\tau(\gamma)_i= T_i(\gamma)-T_i(\gamma_0)$ for $i\in\{1,\ldots,M-1\}$ are at least of order $(\gamma-\gamma_0)^\frac{1}{2}$ and thus $o(|\tau(\gamma)|) = o((\gamma-\gamma_0)^\frac{1}{2})$. Second, set $H = (I_2)^c$.  Whether $T_h(\gamma_0^+)$, $h\in H$ is type-0 or 1, $\tau(\gamma)_h$ is at least $o((\gamma-\gamma_0)^\frac{1}{2})$ (see Eq.~(\ref{eq-Ti_approx_mk1})).  Therefore, the $h$ index of the summation in Eq.~(\ref{eq-Japprox}) are  $(-1)^{\omega_h}\theta\tau(\gamma)_h = o(\gamma-\gamma_0)$.  Finally, plugging Eq.~(\ref{eq-Ti_approx_mk2}) in for each $i\in I_2$ into the summation in Eq.~(\ref{eq-Japprox}) results in 
\[
\begin{array}{l}
(-1)^{\omega_i}\theta\tau(\gamma)_i = - \frac{\sqrt{2}(\theta)^2}{\ddot{d}_{\sigma_{i+\omega_i}}(T_i(\gamma_0^+))^{\frac{1}{2}}}(\gamma-\gamma_0)^{\frac{1}{2}} + o((\gamma-\gamma_0)^{\frac{1}{2}}).
\end{array}
\]
Referring to Eq.~(\ref{eq-hatJk2}), $J(\gamma) = \hat{J}(2,\gamma) + R(\gamma)$ where
\[
\begin{array}{l}
R(\gamma) = \displaystyle{\sum_{i\in I(\overline{m})}}o((\gamma-\gamma_0)^\frac{1}{2}) + \displaystyle{\sum_{h\in H}}o(\gamma-\gamma_0)+ o((\gamma-\gamma_0)^\frac{1}{2}) =  o((\gamma-\gamma_0)^\frac{1}{2}).
\end{array}
\]
Since $\hat{J}(2;\gamma)-J(0)$ is not $o((\gamma-\gamma_0)^\frac{1}{2})$, the lemma is proven.
\end{proof}

\subsection{Proof of Lemma \ref{lem-suff_dec}: Sufficient descent \label{app-suff_dec}}
\begin{proof}
Recall from Eqs.~(\ref{eq-thetacalc}) and (\ref{eq-gamma0k}), $\theta^k = -1/\gamma_0^k = d^k_{\sigma_{i+\omega_i}}(T_i(\gamma_0^{k^+}))<0$ for each $i \in I_2^k$.  Also, according to Lemma~\ref{lem-d_properties}, for each $i\in I^k_2$, there is a neighborhood of $\gamma_0^k$ for which $d^k_{\sigma_{i+\omega_i}}(T_i(\gamma))<0$, $(-1)^{\omega_i}\dot{d}^k_{\sigma_{i+\omega_i}}(T_i(\gamma))>0$ and $\ddot{d}^k_{\sigma_{i+\omega_i}}(T_i(\gamma))>0$.
Set
\[
H(\gamma) := -\alpha \sqrt{2}\textrm{card}(I^k_2)\frac{(\theta^k)^2}{\nu^{\frac{1}{2}}}(\gamma-\gamma_0^k)^\frac{1}{2}.
\]
The right hand side of Eq.~(\ref{eq-t2suff_dec}) is greater than $H(\gamma)$ for all $\gamma>\gamma_0^k$ through the definition of $\nu$.  The proof follows by finding the $\gamma\in(\gamma_0^k,\gamma_1^k]$ for which the derivative of left hand side of Eq.~(\ref{eq-t2suff_dec}) is more negative than the derivative of the right hand side.  The derivative of the left hand side is 
\[
DJ^k(\gamma) = \sum_{i\in I^k_2} (-1)^{\omega_i}\frac{d^{k}_{\sigma_{i+\omega_i}}(T_i(\gamma))^3}{\dot{d}^{k}_{\sigma_{i+\omega_i}}(T_i(\gamma))}
\]
which is negative valued.  The derivative of the right hand side is bounded below by $DH(\gamma)$:
\begin{equation}
DH(\gamma) := -\alpha \frac{\sqrt{2}}{2}\textrm{card}(I^k_2)\frac{(\theta^k)^2}{\nu^{\frac{1}{2}}}(\gamma-\gamma_0^k)^{-\frac{1}{2}}.
\label{eq-DH_bnd}
\end{equation}
The rest of the proof shows $DJ^k(\gamma)<DH(\gamma)$ for all $\gamma\in(\gamma_0^k,\hat{\gamma}^k)$.

Set $\tau_i(\gamma) = T_i(\gamma)-T_i(\gamma_0^k)$.  Since $\ddot{d}^k_{\sigma_{i+\omega_i}}(T_i(\gamma))$ is Lipschitz, by the mean value theorem, 
\[
(-1)^{\omega_i}\dot{d}^k_{\sigma_{i+\omega_i}}(T_i(\gamma)) \leq \ddot{d}^k_{\sigma_{i+\omega_i}}(T_i(\gamma_0^k))\tau(\gamma) - L\tau(\gamma)^2.
\]
Therefore, for $\tau_i(\gamma)\leq\tau_{i,max} := \frac{\ddot{d}^k_{\sigma_{i+\omega_i}}(T_i(\gamma_0^k))}{2L}$
\begin{equation}
(-1)^{\omega_i}\dot{d}^k_{\sigma_{i+\omega_i}}(T_i(\gamma)) \leq\frac{3}{2}\ddot{d}^k_{\sigma_{i+\omega_i}}(T_i(\gamma_0^k))\tau_i(\gamma). 
\label{eq-vdot_bnd}
\end{equation}
By Lipschitz, a lower bound of $\ddot{d}^k_{\sigma_{i+\omega_i}}(T_i(\gamma))$ for $\tau_i(\gamma)\leq\tau_{i,max}$ is 
\[
\begin{array}{l}
\ddot{d}^k_{\sigma_{i+\omega_i}}(T_i(\gamma))\geq\ddot{d}^k_{\sigma_{i+\omega_i}}(T_i(\gamma^{k^+}_0)) + L\tau_i(\gamma) \geq \frac{1}{2}\ddot{d}^k_{\sigma_{i+\omega_i}}(T_i(\gamma^{k^+}_0)).
\end{array}
\]
By the Taylor expansion of $-1-\gamma d^k_{\sigma_{i+\omega_i}}(T_i(\gamma))$ around $T_i(\gamma)$, with remainder $r(T_i(\gamma))$,
\[
\frac{-1}{\gamma} + \frac{1}{\gamma_0^k} + \frac{1}{2}r(T_i(\gamma))\tau_i(\gamma)^2 = 0.
\]
For $\tau(\gamma)<\tau_{i,max}$ the lower bound of $\ddot{d}^k_{\sigma_{i+\omega_i}}(T_i(\gamma))$ is also the lower bound of the remainder term.  In other words, $r(T_i(\gamma))>\frac{1}{2}\ddot{d}^k_{\sigma_{i+\omega_i}}(T_i(\gamma^{k^+}_0))$ and thus for $\tau_i(\gamma)<\tau_{i,max}$,
\begin{equation}
\tau_i(\gamma) \geq \frac{-2 \theta^k}{\ddot{d}^{k}_{\sigma_{i+\omega_i}}(T_i(\gamma_0^{k^+}))^\frac{1}{2}}(\gamma-\gamma_0^k)^\frac{1}{2}.
\label{eq-tau_bnd}
\end{equation}
Indeed, for each $i\in I^k_2$ and $\gamma\in(\gamma_0^k,\min\{\gamma_1^k,\gamma_2^k\}]$, the right hand side of Eq.~(\ref{eq-tau_bnd}) is less than or equal to $\tau_{i,max}$.  Plugging $\gamma_2^k$ into the right hand side of Eq.~(\ref{eq-tau_bnd}) reduces to, 
\[
\begin{array}{l}
\displaystyle{\frac{\nu^\frac{3}{2}}{2L\ddot{d}^{k}_{\sigma_{i+\omega_i}}(T_i(\gamma_0^{k^+}))^\frac{1}{2}}}\leq\frac{\nu}{2L}\leq\tau_{i,max}.
\end{array}
\]
Therefore, Eqs (\ref{eq-vdot_bnd}) and (\ref{eq-tau_bnd}) are true for every  $\gamma\in(\gamma_0^k,\min\{\gamma_1^k,\gamma_2^k\}]$.  For these $\gamma$, an upper bound on $(-1)^{\omega_i}\dot{d}^k_{\sigma_{i+\omega_i}}(T_i(\gamma))$ is
\[
(-1)^{\omega_i}\dot{d}^k_{\sigma_{i+\omega_i}}(T_i(\gamma))\leq-3 \theta^k\ddot{d}^{k}_{\sigma_{i+\omega_i}}(T_i(\gamma_0^{k^+}))^\frac{1}{2}(\gamma-\gamma_0^k)^\frac{1}{2}.
\]
Let $\overline{\nu} = \max_{i\in I^k_2} \ddot{d}^k_{\sigma_{i+\omega_i}}(T_i(\gamma_0^{k^+}))$ and $\psi = \overline{\nu}/\nu>1$. Thus, for each $i\in I^k_2$,
\begin{equation}
(-1)^{\omega_i}\dot{d}^k_{\sigma_{i+\omega_i}}(T_i(\gamma))\leq-3 \theta^k(\nu\psi)^\frac{1}{2}(\gamma-\gamma_0^k)^\frac{1}{2}.
\label{eq-vdotgamma_bnd}
\end{equation}
To find an upper bound on $d^k_{\sigma_{i+\omega_i}}(T_i(\gamma))$, integrate Eq.~(\ref{eq-vdot_bnd}) with respect to $\tau_i(\gamma)$.  
\[
\begin{array}{l}
d^k_{\sigma_{i+\omega_i}}(T_i(\gamma)) < \theta^k + \int_0^{\tau_i(\gamma)}\frac{3}{2}\ddot{d}^k_{\sigma_{i+\omega_i}}(T_i(\gamma_0^k))s\:ds\\\hspace{20pt} = \theta^k + \frac{3}{4}\ddot{d}^k_{\sigma_{i+\omega_i}}(T_i(\gamma_0^k))\tau_i(\gamma)^2
\end{array}
\]
Using the bound in Eq.~(\ref{eq-tau_bnd}) and by setting $\beta(\gamma) = 1+3\theta^k(\gamma-\gamma_0^k)$,
\begin{equation}
d^k_{\sigma_{i+\omega_i}}(T_i(\gamma)) \leq \theta^k\beta(\gamma)
\label{eq-vgamma_bnd}
\end{equation}
With the bounds on $d^k_{\sigma_{i+\omega_i}}(T_i(\gamma))$, Eq.~(\ref{eq-vgamma_bnd}), and $(-1)^{\omega_i}\dot{d}^{k}_{\sigma_{i+\omega_i}}(T_i(\gamma))$, Eq.~(\ref{eq-vdotgamma_bnd}), $DJ^k(\gamma)$ is bounded above by
\begin{equation}
DJ^k(\gamma) \leq -\textrm{card}(I^k_2)\frac{\beta(\gamma)^3}{3}\frac{(\theta^k)^2}{(\nu\psi)^\frac{1}{2}}(\gamma-\gamma_0^k)^{-\frac{1}{2}}.
\label{eq-DJ_bnd}
\end{equation}
Comparing Eqs. (\ref{eq-DH_bnd}) and (\ref{eq-DJ_bnd}), 
\[
\beta(\gamma)^3\geq\alpha\frac{3\sqrt{2}}{2}\psi^\frac{1}{2}\geq\alpha\frac{3\sqrt{2}}{2},
\]
implies $DJ^k(\gamma)<DH(\gamma)$, which is valid for every $\gamma\in\min\{\gamma_1^k,\gamma_2^k,\gamma_3^k\} = \hat{\gamma}^k$.  It follows that each $\gamma\in(\gamma_0^k,\hat{\gamma}^k]$ satisfies the sufficient descent condition.  
\end{proof}

\subsection{Proof of Lemma \ref{lem-min_seq}: Infimizing Sequence \label{app-min_seq}}
\begin{proof}
Consider $(x^{k+1},u^{k+1})\in S_2$ which is calculated from $(x^k,u^{k})$ using backtracking so that $u^{k}-\gamma^{k} d^{k}$ satisfies the type-2 sufficient descent condition,  Eq.~(\ref{eq-t2suff_dec}) and set $\nu^{k} := \min_{i\in I^k_2} \ddot{d}^{k}_{\sigma_{i+\omega_i}}(T_i(\gamma_0^{k^+}))$.  The cost difference between switched system trajectories $(x^{k+1},u^{k+1})$ and $(x^k,u^{k})$ is 
\begin{equation}
J(x^k,u^{k})-J(x^{k+1},u^{k+1}) > \alpha \sqrt{2} \textrm{card}(I^k_2)\frac{(\theta^{k})^2}{(\nu^{k})^\frac{1}{2}}(\gamma^{k}-\gamma_0^{k})^\frac{1}{2}.
\label{eq-J_diff}
\end{equation}
Since $S_2$ has infinite cardinality, it is the case that as $k\rightarrow \infty$, the right hand side of Eq.~(\ref{eq-J_diff}) goes to zero.  By Lemma~\ref{lem-suff_dec} and the assumption on $\gamma_1^{k}$, $\gamma_2^{k}$, and $\gamma_3^{k}$, any $\gamma\in(\gamma_0^{k},\min\{\gamma_2^{k},\gamma_3^{k}\}]$, defined in Lemma~\ref{lem-suff_dec}, satisfies the type-2 sufficient descent condition. Let $L$ be the Lipschitz constant of $\ddot{d}^k_{a}(\cdot)$ for each $a\in\{1,\ldots,N\}$ and every $u^k\in\mathcal{S}_2$.  Recall $\gamma^{k} = (\gamma_3^{k}-\gamma_0^{k})\beta^{j^{k}}+\gamma_0^{k}$ is calculated by backtracking and therefore, if $\gamma_3^{k}\leq\gamma_2^{k}$, then $\beta^{j^{k}} = 0$ and $\gamma^{k} = \gamma^{k}_3$. Conversely, suppose $\gamma_2^{k}<\gamma_3^{k}$.  Due to backtracking, it is possible for $\gamma^k=(\gamma_3^{k}-\gamma_0^{k})\beta^{j^{k}}+\gamma_0^{k}<\gamma_2^k$.  If this is the case, then $(\gamma_3^{k}-\gamma_0^{k})\beta^{j^{k}-1}+\gamma_0^{k}>\gamma_2^k$.  Therefore, $\gamma^k$ is in the interval 
\[
\gamma^k\in[(\gamma_2^k-\gamma_0^k)\beta +\gamma_0^k,\gamma_2^k]
\]
and thus 
\begin{equation}
\gamma^k = \gamma_0^k + \psi^k\frac{(\nu^k)^3}{(\theta^k)^216L^2}
\label{eq-gammak_bnd}
\end{equation}
where $\psi^k\in[\beta,1]$.  By assumptions, it must be the case that there are an infinite number of $u^{k+1}$ calculated from $u^k$ where either 1) $\gamma^k = \gamma^k_3$ or 2) $\gamma^k$ is given by Eq.~(\ref{eq-gammak_bnd}).  Since $\lim_{k\rightarrow \infty}J(x^k,u^{k})-J(x^{k+1},u^{k+1}) = 0$, the limit of the right hand side of Eq.~(\ref{eq-J_diff}) goes to zero. If case 1), then
\[
\lim_{k\rightarrow \infty} \alpha\sqrt{2} \textrm{card}(I^k_2)\left(1-\frac{\sqrt[3]{\alpha \frac{3\sqrt{2}}{2}}}{3}\right)^\frac{1}{2}\frac{(\theta^k)^\frac{3}{2}}{(\nu^k)^\frac{1}{2}} = 0.
\]
Since $\nu^k\leq LT$, $\lim_{k\rightarrow \infty} \theta^k = 0$.  Now, if 2), then
\[
\lim_{k\rightarrow \infty} \frac{\alpha\sqrt{2\psi^k}\textrm{card}(I^k_2)}{4L} \theta^k\nu^k = 0.
\]
Since $\nu^k\geq K_2|\theta^k|$ and $\psi^k\geq \beta>0$, once again, $\lim_{k\rightarrow \infty} \theta^k = 0$ and the proof is complete.
\end{proof}

\bibliographystyle{elsarticle-num}
\bibliography{worksbib}

\end{document}